\documentclass[oneside, a4paper,reqno]{amsart}
\usepackage{pdfsync}
\usepackage{stmaryrd}
\usepackage{mathrsfs}
\usepackage{fancyhdr}
\usepackage{multicol}
\usepackage{amsmath, amsthm, amscd, amssymb, latexsym, eucal}
\usepackage[all]{xy}
\usepackage{stackrel}
\def\serieslogo@{} \def\@setcopyright{} \makeatother

\usepackage{multienum}

\usepackage[colorinlistoftodos]{todonotes}

\usepackage{hyperref}
\usepackage{color}
\usepackage[nospace,noadjust]{cite}

\hypersetup{colorlinks=true,
     breaklinks=true,
     linkcolor=blue,
     citecolor=red,
     bookmarks=true,
     pageanchor=true}

\usepackage{tikz-cd}

\usepackage{pgf,tikz}
\usepackage{tkz-tab}
\usetikzlibrary{shapes,snakes,arrows,backgrounds,calc}

\makeatletter
\renewcommand*\env@matrix[1][c]{\hskip -\arraycolsep
  \let\@ifnextchar\new@ifnextchar
  \array{*\c@MaxMatrixCols #1}}
\makeatother

\usepackage{color}

\numberwithin{equation}{section}
\newtheorem{thm}{Theorem}[section]
\newtheorem*{main-thm}{Main Theorem}

\newtheorem*{Auslandertype}{An Auslander-type ``splitting-big-objects'' Theorem}
\newtheorem*{ARTtheorem}{Auslander-Ringel-Tachikawa Theorem}
\newtheorem*{thmmain}{Theorem}
\newtheorem{cor}[thm]{Corollary}
\newtheorem{lem}[thm]{Lemma}
\newtheorem{prop}[thm]{Proposition}

\theoremstyle{definition}
\newtheorem{defn}[thm]{Definition}
\newtheorem{rem}[thm]{Remark}
\newtheorem{exam}[thm]{Example}

\newtheorem{Problem}[thm]{Problem}

\newcommand{\lxr}{\longrightarrow}

\newcommand{\A}{\mathscr A}
\newcommand{\B}{\mathscr B}
\newcommand{\C}{\mathscr C}
\newcommand{\MM}{\mathscr{M}}
\newcommand{\II}{\mathscr I}
\newcommand{\PP}{\mathscr P}

\newcommand{\M}{\mathcal M}

\newcommand{\T}{\mathcal T}

\newcommand{\X}{\mathcal X}

\newcommand{\mt}{\mathsf{T}}

\newcommand{\mk}{\mathsf{K}}

\newcommand{\mD}{\mathsf{D}}

\DeclareMathOperator*{\Ker}{\mathsf{Ker}}
 
\DeclareMathOperator*{\Coker}{\mathsf{Coker}}

\DeclareMathOperator*{\Mod}{\mathsf{Mod}}

\DeclareMathOperator*{\End}{\mathsf{End}}
 \DeclareMathOperator*{\smod}{\mathsf{mod}}
 \DeclareMathOperator*{\com}{\mathsf{com}}

\DeclareMathOperator*{\Inj}{\mathsf{Inj}}
\DeclareMathOperator*{\Proj}{\mathsf{Proj}}

\newcommand{\GProj}{\operatorname{\mathsf{GProj}}\nolimits}
 \newcommand{\Gproj}{\operatorname{\mathsf{Gproj}}\nolimits}

\DeclareMathOperator*{\CM}{\mathsf{CM}}

\DeclareMathOperator*{\Add}{\mathsf{Add}}
 
\DeclareMathOperator*{\add}{\mathsf{add}}

\DeclareMathOperator{\Hom}{\mathsf{Hom}}
 \DeclareMathOperator{\Rad}{\mathsf{Rad}}

\DeclareMathOperator*{\Ext}{\mathsf{Ext}}

  \DeclareMathOperator*{\op}{\mathsf{op}}

   \DeclareMathOperator*{\Ab}{\A\!\textit{b}}

   \DeclareMathOperator*{\dime}{\mathsf{dim}}

\DeclareMathOperator*{\ke}{\mathsf{Ke}}
\DeclareMathOperator*{\ind}{\mathsf{ind}}

\DeclareMathOperator*{\Cok}{\mathsf{Cok}}
\DeclareMathOperator*{\Epi}{\mathsf{Epi}}

\DeclareMathOperator*{\Mor}{\mathsf{Mor}}

\newcommand{\iden}{\operatorname{Id}\nolimits}

\newsavebox{\proofbox}
\savebox{\proofbox}{\begin{picture}(7,7)%
  \put(0,0){\framebox(7,7){}}\end{picture}}

\newcommand{\lra}{\longrightarrow}

\newcommand{\ra}{\rightarrow}
\newcommand{\rat}{\rightarrowtail}
\newcommand{\hra}{\hookrightarrow}
\newcommand{\tra}{\twoheadrightarrow}
\newcommand{\subs}{\subset}

\begin{document}

\title[]{Exact categories, big Cohen-Macaulay modules and finite representation type}

\author[C. Psaroudakis]{Chrysostomos Psaroudakis and Wolfgang Rump}
{\address{Department of Mathematics, Aristotle University of Thessaloniki, Thessaloniki, 54124, Greece}
\email{chpsaroud@math.auth.gr}}

{\address{Institute of Algebra and Number Theory, University of Stuttgart, Pfaffenwaldring 57, D-70550 Stuttgart, Germany}
\email{rump@mathematik.uni-stuttgart.de}}



\begin{abstract}
One of the first remarkable results in the representation theory of artin algebras, due to Auslander and Ringel-Tachikawa, is the characterization of when an artin algebra is representation-finite. In this paper, we investigate aspects of representation-finiteness in the general context of exact categories in the sense of Quillen. In this framework, we introduce ``big objects'' and prove an Auslander-type ``splitting-big-objects'' theorem. Our approach generalises and unifies the known results from the literature. As a further application of our methods, we extend the theorems of Auslander and Ringel-Tachikawa to arbitrary dimension, i.e.\ we characterise when a Cohen-Macaulay order over a complete regular local ring is of finite representation type.
\end{abstract}

\maketitle

\setcounter{tocdepth}{1} \tableofcontents

\section{Introduction}

In the representation theory of artin algebras, the main problems are concerned with the additive structure of the module category, for instance how modules decompose into a direct sum of indecomposable modules. Finite representation type was one of the first leading issues in this area. The question was when the 
category of finitely generated modules over an artin algebra is of {\em finite 
type}, that is, when the set of isomorphism classes of its indecomposable objects is finite, and in particular what this really means for the whole module category 
itself.

Toward this question, Auslander and Ringel-Tachikawa have proved the following important result. The implication (i)$\Longrightarrow$(ii) below is due to Ringel-Tachikawa \cite{RT}, while the converse (ii)$\Longrightarrow$(i) is {\em Auslander's ``splitting-big-module'' theorem} \cite{Au76}.  

\begin{ARTtheorem}
Let $\Lambda$ be an artin algebra. The following statements are equivalent$\colon$
\begin{enumerate}
\item $\Lambda$ is of finite representation type.

\item Every $\Lambda$-module is a direct sum of finitely generated modules.
\end{enumerate}
\end{ARTtheorem}

In particular, this means that the module category doesn't have large indecomposables, i.e. indecomposables which are not finitely generated.
 
A remarkable analogue of the Auslander-Ringel-Tachikawa theorem was first proved by Chen \cite{Chen}, and later generalised by Beligiannis \cite{Bel}, for the subcategory of Gorenstein-projective modules over an artin algebra $\Lambda$. Beligiannis proved that any Gorenstein-projective $\Lambda$-module is a direct sum of finitely generated ones if and only if the artin algebra $\Lambda$ is virtually Gorenstein of finite Cohen-Macaulay type \cite{Bel}. The latter result improved Chen's theorem \cite{Chen} for Gorenstein algebras. It should be noted that the key idea of Beligiannis' result is an Auslander-Ringel-Tachikawa theorem for a resolving subcategory of the category $\smod\Lambda$ of finitely generated $\Lambda$-modules.

In the commutative case, and for a commutative noetherian local algebra $R$ over a field, Hochster \cite{Hoc} proved the existence of big Cohen-Macaulay modules. Later, Griffith \cite{Gri0} refined Hochster's theorem by showing that over a complete regular local ring $R$, any module-finite domain $S$ with a big Cohen-Macaulay module admits a countably generated one. The major problem in \cite{Gri} was when a countably generated big Cohen-Macaulay module over a complete local Gorenstein ring $R$ splits into a direct sum of finitely generated ones. Griffith's showed that this is the case if  $R$ is representation-finite (\cite[Corollary 5.2]{Gri}). On the other hand, Beligiannis \cite[Theorem~4.20]{Bel} obtained a decomposition theorem for Gorenstein projectives over a complete noetherian commutative local ring $A$, provided that there exists a non-projective finitely generated Gorenstein projective $A$-module. Moreover, there is a remarkable connection between the finiteness of Cohen-Macaulay modules \cite{Bel} and the singularity theory of the ring. In particular, Auslander \cite{Aus:finitetype} proved that every complete Cohen-Macaulay local ring of finite Cohen-Macaulay type is an isolated singularity. In this spirit, Christensen, Piepmeyer, Striuli and Takahashi \cite{CPST} proved, for a commutative noetherian local ring, that  
if the set of indecomposable totally reflexive modules is finite, then either this set has exactly one element or the ring is Gorenstein and an isolated singularity (over a Gorenstein ring the totally reflexive modules are exactly the maximal Cohen-Macaulay modules).

Very recently, the second author established new criteria for 
Cohen-Macaulay finiteness \cite{RumpFCMhdim} which extend the previously known results to arbitrary dimension. The appearance of Cohen-Macaulay finiteness in several different branches of mathematics like commutative algebra, non-commutative singularity theory, Gorenstein homological algebra and related topics, has 
prompted the idea of a common framework for Cohen-Macaulay finiteness, including the non-commutative case. From the work of Beligiannis \cite{Bel} already, but 
manifestly from recent work of the second author \cite{RumpFCMhdim}, it became 
evident that exact categories in the sense of Quillen provide the adequate
setting for tackling that problem. We also refer to the recent work by Enomoto \cite{Enomoto}, where exact categories were used to obtain a complete classification of Cohen-Macaulay finite Gorenstein algebras. 

Motivated by these results, it is natural to explore representation-finiteness with regard to big modules in a general context including higher dimension as well. In this spirit, this paper should be regarded as a natural continuation 
of \cite{RumpFCMhdim}. A first step toward an Auslander-Ringel-Tachikawa Theorem for Krull dimension one, that is, for classical orders over a complete discrete valuation domain, was done successfully by the second author in \cite{Ind}.  
So let us pass to the case of Krull dimension $d\geq 2$. 

Let $R$ be a complete regular local ring of Krull dimension $d$. Recall that a 
{\em Cohen-Macaulay order} $\Lambda$ over $R$ is an $R$-algebra which is finitely generated and free over $R$. Then a $\Lambda$-module $X$ is said to be
{\em Cohen-Macaulay} if $X$ is finitely generated and free over $R$. Note that 
for $d=0$, the order $\Lambda$ is just an artin $R$-algebra, and the category 
$\CM(\Lambda)$ of Cohen-Macaulay modules coincides with the category $\smod\Lambda$ of finitely generated $\Lambda$-modules. The following result 
(Theorem~\ref{t4}), based on a suitable concept of (accessible) {\em big} Cohen-Macaulay module, extends the Auslander-Ringel-Tachikawa theorem to arbitrary 
finite dimension. 

\begin{thmmain}
Let $\Lambda$ be a Cohen-Macaulay order over a complete $d$-dimensional regular 
local ring $R$. The following are equivalent$\colon$
\begin{enumerate}
\item $\Lambda$ is of finite representation type.
\item Every accessible big Cohen-Macaulay $\Lambda$-module is a direct sum of finitely generated $\Lambda$-modules.
\end{enumerate}
\end{thmmain}

The theorem will be obtained as a consequence of an Auslander-type ``splitting-big-objects'' theorem in the context of exact categories. This result is built 
from recent work of the second author \cite{RumpFCMhdim}, where new representation-theoretic techniques are applied to exact categories ($L$-functors, Auslander-Reiten exact categories and other, see subsection~\ref{subsectionacyclic}). To state our general Auslander-type result, we have to introduce some notation first.

We work over {\em Ext-categories} \cite{T}, that is, exact categories with 
enough projectives and injectives, see subsection~\ref{subExtcat}. The first problem in this setup is to define ``big objects'' in a meaningful way. Let $\A$ be an Ext-category. For a full subcategory $\C$ of $\Mod(\PP)$, where $\PP$ denotes the full subcategory of projectives of $\A$, we define $\Add \C$ to be the full subcategory of direct summands of coproducts $\coprod_{\gamma\in\Gamma} C_\gamma$ with $C_\gamma\in\C$. It turns out that $\Add\A$ is
equivalent to ${\bf Add}\A:=\Proj(\Mod(\A))$ (see \cite{Lat}). We construct 
an increasing sequence $\ke_0(\A)\subs\ke_1(\A)\subs\ke_2(\A)\subs\cdots$ of full subcategories of $\Mod\PP$, starting with the category 
$\ke_0(\A)=\Add\A$ of projectives in $\Mod(\A)$. The next step is to consider the 
subcategory $\ke_1(\A)$ of objects $L$ in $\Mod(\PP)$ arising from a short exact 
sequence $0\to L\to M\to A\to 0$ with $M$ and $A$ in $\ke_0(\A)$. 
Iterating this process and taking the union of all $\ke_n(\A)$, we obtain a category $\ke(\A)$ which turns out to be the smallest resolving subcategory of $\Mod(\PP)$ which 
contains $\Add\A$, see Proposition~\ref{p20}. We say that $\ke(\A)$ consists of the ``accessible big'' $\A$-objects, see Definition~\ref{d8} and the remarks after for further explanation.
 
Next, for an additive category 
$\A$ we consider the category $\smod(\A)$ of coherent functors $\A^{\op}\lxr \Ab$ (we view this category as a certain  quotient of the morphism category of $\A$, see subsection~\ref{subhomcattwoterm}). An additive category $\A$ is called {\em strongly left noetherian} \cite{Lat} if $\smod(\A)$ is abelian and noetherian.
 
Our second main result is the following, see Theorem~\ref{t3}.

\begin{Auslandertype}
Let $\A$ be a left Ext-category. Assume that every object of $\ke(\A)$ is a direct sum of objects in $\A$. Then the quotient $\A/[\mathscr{P}]$ is strongly left noetherian.
\end{Auslandertype}

The relevance of ``strongly left noetherian'' is its connection to representation finiteness. Indeed, the latter property together with certain conditions implies that the set of isomorphism classes of the indecomposable $\A$-objects is finite (Theorem~\ref{thmRumpThm1}). This representation-theoretic interpretation is due to the second author and is discussed in subsection~\ref{subsectionacyclic}.

Apart from proving the above Auslander-Ringel-Tachikawa criterion for 
represent\-ation-finite orders in higher dimension, we will use this general {\em ``splitting-big-objects''} theorem to give direct proofs of Beligiannis' results \cite{Bel} (Corollaries~\ref{corBel1},~\ref{corBel2},~\ref{corBel3} and~\ref{corBel4}), and to extend Chen's theorem \cite{Chen} from Gorenstein artin algebras to a wide class of Gorenstein rings (see Corollary~\ref{corChensresult}).

The article is structured as follows. In Section~\ref{sectionprelim} we recall notions and results on exact categories that are used throughout the paper. In Section~\ref{sectionthreederextcat} we discuss the derived category of an Ext-category. In particular, we introduce the notion of an adjoint pair of subcategories (see Definition~\ref{d4}) and show in Proposition~\ref{p15} that the pair $(\smod(\PP), \com(\II))$ is an adjoint pair in the derived category $\mD(\A)$. Here $\smod(\PP)$ and $\com(\II)$ are certain quotients of the category of morphisms over $\A$, see subsection~\ref{subhomcattwoterm}.
At the end of this section, after recalling a notion of a dimension for exact categories introduced by the second author in \cite{RumpFCMhdim}, we characterise in Proposition~\ref{p16} when an Ext-category has finite dimension in terms of the $\A$-resolution dimension of $\smod(\PP)$ in the sense of Auslander-Buchweitz. In Section~\ref{section:totallyacyclic}, working again with an Ext-category $\A$ such that the subcategory of projectives $\PP$ is left coherent, we define the acyclic closure $\mt(\A)$ of $\A$. We first show in Proposition~\ref{p17} that the acyclic closure $\mt(\A)$ is again an Ext-category. Moreover, we show in Proposition~\ref{p18} that $\mt(\A)$ consists of the Gorenstein-projectives in $\smod(\PP)$. As a consequence, we derive in Corollary~\ref{corexactacyclic} that a complex over $\A$ is exact if and only if it is acyclic over $\mt(\A)$.  This clarifies the relationship between acyclicity with exactness. Based on this we call $\A$ totally acyclic (Definition~\ref{d7}) if $\mt(\A)=\A$. The final Section~\ref{sectionAuslandertyperesult} is devoted to show the main results of the paper as presented in the first part of the introduction.

We remark that our approach to finite representation type is again, in a sense, ``functorial'' but with many homological influences now in the context of exact categories based on the work \cite{RumpFCMhdim}. For an overview of Auslander's functorial approach on finite representation type, we refer the reader to the book of Krause \cite{Krause:book}, see also \cite{KrauseVossieck}. We also refer to the book of Leuschke-Wiegand \cite{CMrepresentations} for an overview on Cohen-{M}acaulay representations.

\subsection*{Conventions and Notation}
For a ring $R$ we work usually with left $R$-modules and the
corresponding category is denoted by $\Mod(R)$.  The full subcategory
of finitely presented $R$-modules is denoted by $\smod(R)$. Our
additive categories are assumed to have finite direct sums and our
subcategories are assumed to be closed under isomorphisms and direct
summands. If $\X$ is a full subcategory of an abelian category $\A$, we denote by 
 $\Add{\X}$ (respectively, $\add{X}$) the full subcategory of $\A$ consisting of all objects which are summands of a direct sum (respectively, finite direct sum) of objects of $\X$. The Jacobson radical of a ring $R$ is denoted by $\Rad R$.
By a module over an artin algebra $\Lambda$, we mean a finitely
presented (generated) left $\Lambda$-module. We also write
$\mathscr{P}:=\Proj(\A)$ and $\mathscr{I}:=\Inj(\A)$ for the projective, respectively injective, objects of the category $\A$.


\section{Preliminaries}
\label{sectionprelim}

\subsection{$\Ext$-Categories}
\label{subExtcat}
Let $\A$ be an additive category. A pair of morphisms in $\A\colon$
\[
\begin{tikzcd}
A \arrow[r, "a"] & B \arrow[r, "b"] & C  
\end{tikzcd}
\]
is said to be a {\em short exact sequence} if $a=\Ker{b}$ and $b=\Coker{a}$. An additive category $\A$ is called {\em exact} if there is a non-empty class $\mathsf{Con}(\A)$ of short exact sequences satisfying certain axioms. We assume that $\mathsf{Con}(\A)$ is closed under isomorphisms. Following Keller \cite{Ke}, the short exact sequences in $\mathsf{Con}(\A)$ are called conflations, the morphism $a$ (respectively, $b$) in a conflation as above is called inflation (respectively, deflation). Then the defining axioms of an exact category $\A$ are the following$\colon$
\begin{enumerate}
\item The composition of inflations (respectively, deflations) is an inflation (respectively, deflation).

\item The pullback (respectively, pushout) of a deflation (respectively, an inflation) along an arbitrary morphism exists and is a deflation (respectively, an inflation).
\end{enumerate}
Exact categories were introduced by Quillen \cite{Qu}. We refer to \cite{Buhler} for an overview of the basic homological theory in exact categories in the sense of Quillen.

We now recall the notion of an Ext-category introduced by the second author in \cite{T}. Let $\C$ be a full subcategory of an additive category $\A$. Denote by $\Epi(\C)$ the class of $\C$-epimorphisms, that is, morphisms $e\colon A\lxr A'$ such that any morphism $C'\lxr A'$ with $C'$ in $\C$ factors through $e$. Recall from \cite{AuSm} that $\C$ is contravariantly finite in $\A$ if every object $A$ in $\A$ admits a $\C$-epimorphism $C\lxr A$ with $C$ in $\C$. On the other hand, given any class of morphisms $\Sigma\subset \A$ we can form the full subcategory of $\Sigma$-projectives in $\A$ denoted by $\mathsf{Pr}\Sigma$. This means precisely that $\mathsf{Pr}\Sigma$ is the largest full subcategory $\C$ of $\A$ such that $\Sigma\subset \Epi(\C)$. Motivated by the work of Maranda \cite{M}, a pair $(\C,\Sigma)$ is called a {\em projective structure} in $\A$ if $\C=\mathsf{Pr}\Sigma$, $\Sigma=\Epi(\C)$ and for every $A$ in $\A$ there exists a morphism $C\lxr A$ with $C$ in $\C$.

We refer to the reader to \cite{TrE} for characterizing when a full subcategory of $\A$ gives rise to a projective structure as well as when a morphism class in $\A$ defines a projective structure. We also leave to the reader to formulate the duals of the above concepts ($\C$-monomorphism, $\Sigma$-injective and injective structure).

Suppose now that $\A$ is an exact category. Note first that by the definition it follows that the split short exact sequences belong to $\mathsf{Con}(\A)$. Denote by $\mathsf{Def}(\A)$, resp. $\mathsf{Inf}(\A)$, the class of deflations, resp. inflations, in $\A$. We follow the standard notation for an inflation and a deflation, i.e.\  $\rat$ and $\tra$ respectively. The objects of $\Proj(\A):=\mathsf{Pr}(\mathsf{Def}(\A))$ are called {\em projective} and the objects of $\Inj(\A):=\mathsf{In}(\mathsf{Inf}(\A))$ are called {\em injective}. If every object $A$ in $\A$ admits a deflation $P\lxr A$ with $P$ in $\Proj\A$, then we say that $\A$ has enough projectives. Dually, we say that $\A$ has enough injectives if $\A^{\op}$ has enough projectives. 

Recall from \cite{RumpLfunctors, TrE} that an exact category $\A$ is called {\em divisive} if every split epimorphism has a kernel. We recall the following characterization for an exact category to be divisive. For the proof see \cite[Proposition~1]{RumpLfunctors}.

\begin{lem}
\label{lemdivisive}
Let $\A$ be an exact category. The following are equivalent$\colon$
\begin{enumerate}
\item $\A$ is divisive.

\item A morphism $b\colon B\lxr C$ is a deflation in $\A$ whenever there is a morphism $a\colon A\lxr B$ such that $ba$ is a deflation.
\end{enumerate}
\end{lem}

For an exact category $\A$, the pair $(\Proj(\A), \mathsf{Def}(\A))$ is a projective structure if and only if the exact category $\A$ is divisive and has enough projective objects \cite{RumpFCMhdim}. An exact category with the latter property is called a {\bf left Ext-category}. Dually we have the notion of a right Ext-category. A divisive exact category with enough projectives and enough injectives is called an {\bf Ext-category}.

\subsection{The Homotopy Category of Two-Term Complexes}
\label{subhomcattwoterm}
Let $\A$ be an additive category. We denote by $\Mor(\A)$ the category of morphisms over $\A$. The objects of $\Mor(\A)$ are two-termed complexes $0\lxr A_1\lxr A_0\lxr 0$ and given another object $0\lxr B_1\lxr B_0\lxr 0$ in $\Mor(\A)$, a morphism between these complexes is given by a pair of maps $(f,g)$ such that the following square commutes$\colon$
\[
\begin{tikzcd}
A_1 \arrow[d, "f"] \arrow[r, "a"] & A_0 \arrow[d, "g"]   \\
B_1 \arrow[r, "b"] & B_0  
\end{tikzcd}
\]
Moreover, there is a natural fully faithful functor $\A\lxr \Mor(\A)$ given by the assignment  $A\mapsto \iden_A\colon A\lxr A$. Denote by $[\A]$ the ideal of $\Mor(\A)$ generated by the identity morphisms $\iden_{A}$ in $\Mor(\A)$. It is easy to check that the ideal $[\A]$ of $\Mor(\A)$ consists of homotopic to zero morphisms. We denote by $\mathsf{M}(\A)$ the homotopy category of $\Mor(\A)$ which is equivalent to the quotient category $\Mor(\A)/[\A]$. There are two natural full embeddings$\colon$
\[
\begin{tikzcd}
\A^{+} \arrow[r, hook] & \mathsf{M}(\A) & \arrow[l, hook] \A^{-}
\end{tikzcd}
\]
where $\A^+$, resp. $\A^{-}$, consists of the morphisms $A^+\colon 0\lxr A$, resp. $A^{-}\colon A\lxr 0$, with $A$ in $\A$. Then the factor category $\smod(\A):=\mathsf{M}(\A)/[\A^{-}]$ is equivalent to the category of coherent functors $\A^{\op}\lxr \Ab$ (see \cite[Chapter III.1]{Au71}). For the notion of coherent functor we refer to \cite{Au71}. Moreover, the quotient category $\mathsf{M}(\A)/[\A^{+}]$ is equivalent to $\com(\A):=(\smod(\A^{\op}))^{\op}$.

Let $\A$ be an exact category. We also need to consider the full subcategory $\Ext(\A)$ of the homotopy category $\mathsf{K}(\A)$ of complexes over $\A$ consisting of three-termed complexes $0\lxr A_0\lxr A_1\lxr A_2\lxr 0$ which are conflations in $\A$. In fact, we have the following full embeddings$\colon$
\[
\begin{tikzcd}
\com(\A) & \Ext(\A) \arrow[l, hook] \arrow[r, hook] & \smod(\A) 
\end{tikzcd}
\]
\[
 \ \mathsf{inflation} \ \longmapsfrom \ \mathsf{conflation} \ \longmapsto 
\  \mathsf{deflation}
\]

We need in the sequel the following interesting result on the category $\Ext(\A)$, for the proof see \cite[Proposition~2]{RumpFCMhdim}. 

\begin{prop}
\label{RumpProp2Extab}
Let $\A$ be a divisive exact category. Then $\Ext(\A)$ is abelian. Moreover, if $\A$ is a left $Ext$-category, then $\Ext(\A)\simeq \smod(\underline{\A})$ with $\underline{\A}:=\A/[\Proj\A]$.
\end{prop}

\subsection{Acyclic Complexes}
\label{subsectionacyclic}

Let $\A$ be an exact category and consider a complex
\begin{equation}
\label{2}
\begin{tikzcd}
A\colon \ \cdots \arrow{r}  & A_{-1} \arrow[r, "a_{-1}"] & A_0 \arrow[r, "a_0"] & A_1 \arrow[r, "a_1"] & A_2 \arrow{r}{} & \cdots
\end{tikzcd}
\end{equation}
The complex $A$ is called {\em acyclic} if there exist conflations in $\A\colon$
\[
\begin{tikzcd}
Z_{n-1} \arrow[r, rightarrowtail, "i_{n-1}"] & A_n \arrow[r, twoheadrightarrow, "p_n"] & Z_n 
\end{tikzcd}
\]
with $a_n=i_np_n$ for all $n\in\mathbb{Z}$. 

We will need throughout the paper the following useful result due to Keller. Note that the first statement is the dual of (\cite[Lemma~4.1]{Ke}).

\begin{lem}
\label{Keller'slemma}
Let $\A$ be an exact category.
\begin{enumerate}
\item Assume that $\A$ has enough projectives. Then for each bounded above complex $A$, there is a triangle in the homotopy category $\mathsf{K}^{-}(\A)$ of bounded above complexes$\colon$
\[
\begin{tikzcd}
Z[-1] \arrow{r} & P \arrow{r} & A \arrow{r} & Z 
\end{tikzcd}
\]
such that $Z$ is acyclic and each component in $P$ is projective.

\item Assume that $\A$ has enough injectives. Then for each bounded below complex $B$, there is a triangle in the homotopy category $\mathsf{K}^{+}(\A)$ of bounded below complexes$\colon$
\[
\begin{tikzcd}
 Z \arrow[r]  & B \arrow[r] & I \arrow[r] & Z[1]
\end{tikzcd}
\]
such that $Z$ is acyclic and each component in $I$ is injective.
\end{enumerate}
\end{lem}

We also need the following standard result in the context of exact categories.

\begin{lem}
\label{lemisomderhomcat}
Let $\A$ be an exact category. Let $A$ be a complex in $\A$ and $P$ a bounded above complex of projectives. Then $\Hom_{\mD(\A)}(P,A)\cong \Hom_{\mk(\A)}(P,A)$.
\end{lem}

We call $\A$ {\em left coherent} if every morphism $f$ in $\A$ has a 
{\em weak kernel}, that is, a morphism $g$ in $\A$ with $fg=0$ such that every $g'$ in $\A$ with $fg'=0$ factors through $g$. {\em Right coherence}, and {\em weak cokernels} are defined dually. Left and right coherent additive categories are said to be {\em coherent}. 

A finite or infinite sequence of morphisms 
\begin{equation}
\label{2}
\begin{tikzcd}
\cdots \arrow{r}  & A_{-1} \arrow[r, "a_{-1}"] & A_0 \arrow[r, "a_0"] & A_1 \arrow[r, "a_1"] & A_2 \arrow{r}{} & \cdots
\end{tikzcd}
\end{equation} 
is called {\em weak exact} if $a_n$ is a weak kernel of $a_{n+1}$ and $a_{n+1}$ is a weak cokernel of $a_n$ for all possible $n$.

An additive category $\A$ is called {\bf strongly left noetherian} \cite{Lat} if $\smod{\A}$ is abelian and noetherian. Recall that $\smod(\A)$ being noetherian means that the subobjects of any object of $\smod(\A)$ satisfy the ascending chain condition. The property of $\A$ being strongly left noetherian has been characterized in \cite[Proposition~2]{Lat}. In particular, it is equivalent to the property that for every family $(f_i)_{i\in I}$ of morphisms $f_i\colon A_i\lxr A$ in $\A$, there is a finite subset $J\subset I$ such that each $f_i$ factors through the morphism $\oplus_{j\in J}A_j\lxr A$. If $\smod{\A^{\op}}$ is abelian and noetherian, then $\A$ is called {\bf strongly right noetherian} and $\A$ is {\bf strongly noetherian} if it is strongly left noetherian and strongly right noetherian.

Recall that an additive category $\A$ is Krull-Schmidt if every object of $\A$ decomposes into a finite direct sum of objects with local endomorphism rings. The decompositions into indecomposable objects are unique, up to isomorphism, and the radical $\Rad(\A)$ is generated by the non-invertible morphisms between indecomposables. We denote by $\ind(\A)$ a representative system of indecomposable objects. Following the terminology of \cite{RumpFCMhdim}, a projective and injective object in an exact category is called {\em bijective}. Such an object $B$ is called {\em tame} if there is either a right almost split map ending in $B$ or a left almost split map starting from $B$. We now recall the latter notions. From \cite[Corollary~2]{RumpFCMhdim} we know that any morphism $f\colon A\lxr B$ in $\A$ is of the form $(0 \ m)\colon A_0\oplus A_1\lxr B$ where the morphism $m$ is right minimal and the decomposition $A=A_0\oplus A_1$ is unique up to isomorphism. Recall that a morphism $m\colon A_1\lxr B$ is {\em right minimal} \cite{ARS} if every endomorphism $g\colon A_1\lxr A_1$ with $mg=m$ is invertible. Also, a morphism $f\colon A\lxr B$ in $\A$ is called {\em right almost split} if $f$ is a radical morphism and every other radical morphism $A'\lxr B$ factors through $f$. Thus a right almost split map $f$ is of the form $(0 \ m)$ with $m$ right minimal and right almost split. In this case the morphism $m$ is called a {\em sink map}. Dually we define {\em left almost split morphisms} and {\em source maps}.

One of the key ideas for the criteria of Cohen-Macaulay finiteness in higher dimension \cite{RumpFCMhdim} by the second author is the notion of $L$-functors. Roughly speaking, the second author has proved that the existence of almost split sequences in an exact category $\A$ implies the existence of an adjunction $(L, L^{-})$ on the homotopy category $\mathsf{M}(\A)$. The latter adjunction gives rise to an augmentation morphism $\lambda\colon L\lxr \iden_{\A}$ and given an object $a$ in $\mathsf{M}(\A)$, an iterated application of $L$ gives what is called a {\em left ladder}$\colon$$\cdots\lxr L^2a\lxr La\lxr a$. If every such left ladder is finite, then $\mathsf{M}(\A)$ is called {\em left $L$-finite}. Using the right adjoint $L^{-1}$ we have the notion of {\em right ladder} and $\mathsf{M}(\A)$ being {\em right $L$-finite}. When $\mathsf{M}(\A)$ is left and right $L$-finite then $\mathsf{M}(\A)$ is called {\em $L$-finite}. For more details we refer to \cite{T, RumpLfunctors, add, RumpFCMhdim}. 

We can now state the following result due to second author, see \cite[Theorem~1]{RumpFCMhdim}.

\begin{thm}
\label{thmRumpThm1}
Let $\A$ be an Ext-category with the Krull-Schmidt property. Assume that $\ind\PP$ and $\ind\II$ are finite, that $\End_{\A}(A)$ is noetherian for all objects $A$ in $\A$, and that the indecomposable bijectives are tame. The following are equivalent$\colon$
\begin{enumerate}
\item The number $\ind(\A)$ is finite.

\item $\A$ is strongly noetherian.

\item $\A/[\PP]$ is strongly left noetherian and $\A/[\II]$ is strongly right noetherian.

\item $\Ext(\A)$ is a length category.

\item $\A$ has almost split sequences and $\mathsf{M}(\A)$ is $L$-finite.
\end{enumerate}
\end{thm}

As a consequence of the above, we have the following important result (see \cite[Corollary~1]{RumpFCMhdim}) which is used extensively in Section~\ref{sectionAuslandertyperesult}.

\begin{cor}
\label{CorRumpCor1}
Let $\A$ be a left Ext-category with left almost split sequences such that $\End_{\A}(A)$ is right noetherian for all objects $A$ in $\A$. Consider the following statements$\colon$
\begin{enumerate}
\item The number $\ind(\A)$ is finite.

\item $\A$ is strongly left noetherian.

\item $\A/[\PP]$ is strongly left noetherian.

\item $\Ext(\A)$ is noetherian.

\item $\Ext(\A)$ is a length category.

\item $\A$ has almost split sequences and $\mathsf{M}(\A)$ is left $L$-finite.
\end{enumerate}
Then \textnormal{(i)}$\Longrightarrow$\textnormal{(ii)}$\Longrightarrow$\textnormal{(iii)}$\Longrightarrow$\textnormal{(iv)}
$\Longrightarrow$\textnormal{(v)}$\Longrightarrow$\textnormal{(vi)}.
\end{cor}

We close this subsection with the following important result. In particular, this is \cite[Theorem~2]{RumpFCMhdim} and is one of the key ingredients for proving Beligiannis representation-finiteness result for subcategories, see Corollary~\ref{corBelartinianring}. Recall from \cite{RumpFCMhdim} that a Krull-Schmidt exact category $\A$ is called {\em Auslander-Reiten} if every object $A$ in $\A$ admits a sink and a source map.

\begin{thm}
\label{thmRumpThm2}
Let $\A$ be a left Ext-category with almost split sequences and $\PP:=\Proj(\A)$ such that $\ind(\PP)$ is finite, $\End_{\A}(A)$ is noetherian for all objects $A$ in $\A$, and that indecomposable bijectives are tame. Assume that $\mathsf{M}(\A)$ is left $L$-finite. Consider the following statements$\colon$
\begin{enumerate}
\item $\mathsf{M}(\A)$ is $L$-finite with $\ind(\Inj(\A))$ finite.

\item $\A$ is an Auslander-Reiten category and $\A$, viewed as a subcategory of $\smod(\PP)$, is contravariantly finite.

\item The stable category $\A/[\PP]$ has a weak cogenerator. 

\item $\ind(\A)$ is finite.
\end{enumerate}
Then$\colon$\textnormal{(i)$\Longrightarrow$(ii)$\Longrightarrow$(iii)$\Longrightarrow$(iv)}.
\end{thm}

\subsection{Gorenstein-Projective Objects}
\label{subsectiongorprojectives}
In this last subsection, we briefly recall the notion of Gorenstein-projective modules over a ring \cite{EJ} that is used later in the paper, and we fix notation.

Let $R$ be a ring. An acyclic complex of projective $R$-modules   
\[
\begin{tikzcd}
P^{\bullet}\colon \  \cdots \arrow[r]  & P_{-1} \arrow[r] & P_0 \arrow[r] & P_1 \arrow[r] & \cdots
\end{tikzcd}
\]
is called totally acyclic, if the complex $\Hom_{R}(P^{\bullet}, P)$ is acyclic for every projective $R$-module $P$. An $R$-module $M$ is called {\em Gorenstein-projective} if it is of the form $M=\Coker(P_{-1}\lxr P_0)$ for some totally acyclic complex $P^{\bullet}$ of projective $R$-modules. We denote by $\GProj(R)$ the full subcategory of $\Mod(R)$ consisting of the Gorenstein projective left $R$-modules. By $\Gproj(R)$ we denote the subcategory of the finitely generated Gorenstein projectives.

It is well known that the category $\GProj(R)$ of Gorenstein projectives is a Frobenius exact category with coproducts. Hence, the category $\GProj(R)$ is an Ext-category with coproducts and the projectives coincides with the injectives. Similarly, the subcategory $\Gproj(R)$ is an Ext-category. We remark that the above definition of Gorenstein projectives makes sense for any abelian category with enough projectives. The abelian version of the latter notion is used in Proposition~\ref{p18} where we consider the subcategory $\GProj{(\smod(\PP))}$.

\section{The derived category of an Ext-category}
\label{sectionthreederextcat}

Let $\A$ be an exact category. We denote by $\mathsf{Ac}(\A)$ the full triangulated subcategory of the homotopy category $\mathsf{K}(\A)$ consisting of all complexes which are isomorphic to acyclic complexes, see subsection~\ref{subsectionacyclic}. The derived category $\mD(\A)$ of $\A$ is defined to be the localization $\mk(\A)/\mathsf{Ac}(\A)$ (see Neeman \cite{Nee} and Keller \cite{Kel}).

For any $n\in\mathbb{Z}$, the objects of $\mD(\A)$ isomorphic to complexes $A$ with $A_m=0$ for $m>n$ determine a full subcategory $\mD^{\le n}(\A)$ of $\mD(\A)$. Similarly, we define $\mD^{\ge n}(\A)$ and write $\mD^{<n}(\A):=\mD^{\le n-1}(\A)$ and $\mD^{>n}(\A):=\mD^{\ge n+1}(\A)$. The same notation will be applied to $\mk(\A)$.

We make the following definition.

\begin{defn}
\label{d3}
Let $\A$ be an Ext-category and $n$ an integer. 
\begin{enumerate}
\item A complex (\ref{2}) over $\A$ is called {\bf left exact at $n$} if every map $f\colon P\lxr A_n$ with $a_nf=0$ and $P\in\Proj(\A)$ factors through $a_{n-1}$.

\item A complex (\ref{2}) over $\A$ is called {\bf right exact at $n$} if every map $g\colon A_n\lxr I$ with $ga_{n-1}=0$ and $I\in\Inj(\A)$ factors through $a_n$. 

\item We say that (\ref{2}) is {\bf right exact} (resp. {\bf left exact}) if it is right (resp. left) exact at all $n\in\mathbb{Z}$. If the complex is left and right exact, we call it {\bf exact}. 
\end{enumerate}
\end{defn}

Clearly, every acyclic complex is exact. In the derived category $\mD(\A)$, exactness can be detected as follows. Recall that we write $\mathscr{P}:=\Proj(\A)$ and $\mathscr{I}:=\Inj(\A)$ for the projective and injective objects, repsectively. 

\begin{prop}
\label{p12}
Let $\A$ be an Ext-category. Assume that $\mathscr{P}$ is left coherent. For a complex $A$ over $\A$, the following statements are equivalent$\colon$
\begin{enumerate}
\item $A$ is left exact at $n\leq 0$.

\item $\Hom_{\mk(\A)}(P,A)=0$ for every $P\in\mk^{\leq 0}(\PP)$. 

\item $\Hom_{\mD(\A)}(B,A)=0$ for every $B\in\mD^{\leq 0}(\A)$. 
\end{enumerate}        
\begin{proof}
(i)$\Longrightarrow$(ii): Since $A$ is left exact at $n\leq 0$, it follows easily using Definition~\ref{d3} that any morphism of complexes $P\lxr A$ is homotopic to zero.

(ii)$\Longrightarrow$(iii): For any $B$ in $\mk^{\leq 0}(\A)$, by Lemma~\ref{Keller'slemma} there is a triangle  
\[
\begin{tikzcd}
Z[-1] \arrow{r} & P \arrow{r} & B \arrow{r} & Z 
\end{tikzcd}
\]
with acyclic $Z\in\mk^{\leq 0}(\A)$ and $P\in \mk^{\leq 0}(\PP)$. Hence, Lemma~\ref{lemisomderhomcat} implies that $\Hom_{\mD(\A)}(A,B)=0$.

(iii)$\Longrightarrow$(i): Take $B$ to be a stalk complex at some $n$ over $\PP$. By Lemma~\ref{lemisomderhomcat}, any morphism in $\Hom_{\mD(\A)}(B,A)$ can be represented by a morphism in $\mk(\A)$. Hence $\Hom_{\mk(\A)}(B,A)=0$. This shows that $A$ is left exact at $n\leq 0$.
\end{proof}
\end{prop} 

Let $\mD_{\le n}(\A)$ denote the full subcategory of 
$\mD(\A)$ consisting of the complexes which are right exact at $m>n$. Similarly, $\mD_{\ge n}(\A)$ stands for the full subcategory of
complexes over $\A$ which are left exact at $m<n$. Thus 
$\mD^{\le n}(\A)\subs\mD_{\le n}(\A)$ and 
$\mD^{\ge n}(\A)\subs\mD_{\ge n}(\A)$. As usual, if $\C$ is a full subcategory of a triangulated category $\T$, the right Hom-orthogonal subcategory $\{A\in \T \ | \ \Hom_{\T}(\C,A)=0\}$ of $\T$ is denoted by $\C^\perp$, and the left Hom-orthogonal subcategory $\{A\in \T \ | \ \Hom_{\T}(A,\C)=0\}$ of $\T$ is denoted by ${}^\perp\C$. By Proposition~\ref{p12} and its
dual, we have
\begin{equation}
\label{11}
\mD_{\le n}(\A)={}^\perp\mD^{>n}(\A),\qquad  
\mD_{\ge n}(\A)=\mD^{<n}(\A)^\perp.
\end{equation} 

Assume that the additive category $\mathscr{P}$ is left coherent, that is, $\smod(\mathscr{P})$ is an abelian category (see \cite[Theorem~3.2]{Fr}). We call a full subcategory $\C$ of $\smod(\mathscr{P})$ with $\mathscr{P}\subs\C$ {\em resolving} \cite{AuBr} if for any short exact sequence $0\lxr L\lxr M\lxr N\lxr 0$ in $\smod(\mathscr{P})$ with $N\in\C$, the middle term $M$ is isomorphic to an object in $\C$ if and only if $L$ is isomorphic to an object in $\C$. Auslander and Bridger \cite{AuBr} also require that $\C=\add\C$ which we do not assume. In particular, we show below that a resolving subcategory inherits an exact structure from $\smod(\mathscr{P})$ (see also \cite[Proposition~14]{RumpFCMhdim}).

\begin{prop}
\label{p13}
Let $\A$ be an Ext-category with $\mathscr{P}:=\Proj(\A)$ left coherent. 
Then there are full embeddings 
\[
\begin{tikzcd}
\A \arrow[r, hook] & \smod(\PP) \arrow[r, hook] & \mD(\A)
\end{tikzcd}
\]
such that $\A$ becomes a resolving subcategory of $\smod(\mathscr{P})$ with the induced exact structure, and there is an equivalence
\begin{equation}
\label{12}
\smod(\mathscr{P})\simeq \mD^{\le 0}(\A)\cap\mD_{\ge 0}(\A).  
\end{equation} 
\begin{proof}
For any object $A\in\A$ there is a conflation $A'\stackrel{i}{\rat} P_0
\stackrel{p}{\tra} A$ with $P_0\in\Proj(\A)$, and a deflation $q\colon
P_1\lxr A'$ with $P_1\in \Proj(\A)$. Then $iq\colon P_1\lxr P_0$ defines an object $J(A)\in\smod(\mathscr{P})$. This gives a well-defined additive functor $J\colon\A\lxr\smod(\mathscr{P})$ which is fully faithful. By the Horseshoe lemma, see \cite[Theorem~12.8]{Buhler}, it is straightforward to verify that $J$ maps conflations in $\A$ to short exact sequences in $\smod(\mathscr{P})$.

Let $L\stackrel{a}{\rat} M\stackrel{b}{\tra} B$ be a short exact sequence 
in $\smod\mathscr{P}$ with $B$ in $\A$. Assume first that $M$ lies in $\A$. Since $\A$ is divisive and every morphism $P\lxr B$ with $P\in\mathscr{P}$ factors through $b$, Lemma~\ref{lemdivisive} implies that $b$ is a deflation in $\A$. Since $J$ is exact, we infer that $L$ is isomorphic to an object in $\A$. 

Now assume that $L$ is in $\A$. Choose an epimorphism $p\colon P\lxr M$ in $\smod(\mathscr{P})$ with $P\in\mathscr{P}$. Since every morphism $Q\lxr B$ with $Q\in\mathscr{P}$ factors through $bp$, Lemma~\ref{lemdivisive} implies that $bp$ is a deflation in $\A$. So we get a pullback diagram$\colon$
\[
\begin{tikzcd}
C \arrow[d, rightarrowtail, "j"] \arrow[r, equal] & C \arrow[d, rightarrowtail, "ij"] &  \\
A \arrow[r, rightarrowtail, "i"] \arrow[d, twoheadrightarrow, "q"] & P \arrow[d, twoheadrightarrow, "p"] \arrow[r, twoheadrightarrow] & B \arrow[d, equal] \\
L \arrow[r, rightarrowtail, "a"] & M \arrow[r, twoheadrightarrow, "b"] & B 
\end{tikzcd}
\]
with an inflation $i\in\A$. Since every $Q\lxr L$ with $Q\in\mathscr{P}$ factors through $q$, we infer as above that $q$ is a deflation in $\A$. Hence the composition $ij$ is an inflation, which shows that $M$ is isomorphic to an object in $\A$. This proves that $\A$ is a resolving subcategory of $\smod(\mathscr{P})$. Moreover, it follows that the exact structure of $\A$ is induced by that of $\smod(\mathscr{P})$.

By Lemma~\ref{Keller'slemma}, the full embedding 
$\smod(\mathscr{P})\hra \mD(\A)$ which associates a projective resolution to any object of $\smod(\mathscr{P})$ verifies (\ref{12}). 
\end{proof}
\end{prop} 

We leave it to the reader to dualize the results of this section. For example, the right-hand analogue of (\ref{12}) is
\[
\com(\II)\simeq \mD^{\ge 0}(\A)\cap \mD_{\le 0}(\A).
\]

\begin{defn}\label{d4}
Let $\A$ be an additive category with full subcategories $\B$ and $\C$ 
which are closed under isomorphism. Consider the following diagram
\[
\begin{tikzcd}
 & \A &   \\
\B \arrow[ur, hook] \arrow[rr, "G"]  &  & \C \arrow[ll, bend left, "H"] \arrow[ul, hook] 
\end{tikzcd}
\]
where $G$ and $H$ are additive functors. We call $(\B,\C)$ an {\bf adjoint pair}  if every object $B\in\B$ admits a morphism $\lambda_B\colon B\lxr GB$ with  $GB\in\C$ such that each $B\lxr C'$ with $C'\in\C$ factors uniquely through $\lambda_B$, and every object $C\in\C$ admits a morphism $\rho_C\colon HC\lxr C$ with $HC\in\B$ such that each $B'\lxr C$ with $B'\in\B$ factors uniquely 
through $\rho_C$. 
\end{defn}

For example, a reflective full subcategory $\C\hra\A$, i.e. the inclusion has a left adjoint, is equivalent to an
adjoint pair $(\A,\C)$. In general, every adjoint pair $(\B,\C)$ gives
rise to an adjunction $G\dashv H$ between $\B$ and $\C$. Indeed, any pair of
objects $B\in\B$ and $C\in\C$ determines a commutative diagram
\begin{equation}\label{13}
\begin{tikzcd}
B \arrow[d, "f'"] \arrow[r, "\lambda_B"] & GB \arrow[d, "f"]   \\
HC \arrow[r, "\rho_C"] & C  
\end{tikzcd}
\end{equation}
which gives a bijection $f\mapsto f'$ between $\Hom_{\C}(GB,C)$ and 
$\Hom_{\B}(B,HC)$. The unit $\eta\colon \iden_{\B}\lxr HG$ of the adjunction satisfies $\rho G\circ\eta =\lambda$.

More generally, we say that $(\C_0,\ldots,\C_n)$ is an {\bf adjoint sequence} of full subcategories $\C_i$ if the $(\C_i,\C_{i+1})$ with $0\le i<n$ are adjoint pairs such that each morphism $C_i\lxr C_j'$ with 
$C_i\in\C_i$, $C_j'\in\C_j$ and $0\le i<j\le n$ factors uniquely through the 
composed morphism 
\[
\begin{tikzcd}
C_i \arrow[r, "\lambda_{C_i}"] & C_{i+1} \arrow[r] & \cdots \arrow[r] & C_{j-1} \arrow[r, "\lambda_{C_{j-1}}"] & C_j.  
\end{tikzcd}
\]
Note that the latter condition is left-right symmetric so that $(\C_i,\C_j)$ 
is an adjoint pair for all $0\le i<j\le n$.

\begin{defn}
\label{d5}
Let $(\B,\C)$ and $(\B',\C')$ be adjoint pairs in an additive 
category $\A$. We say that $(\B',\C')$ {\bf extends} $(\B,\C)$ if 
$(\B',\B,\C,\C')$ is an adjoint sequence where $\B$ is a reflective 
subcategory of $\B'$ and $\C$ is a coreflective subcategory of $\C'$.
\end{defn} 

Before we apply these concepts to the derived category, we note the following.

\begin{prop}
\label{p14}
Any extension $(\B',\C')$ of an adjoint pair $(\B,\C)$ in an additive 
category satisfies $\B'\cap\C'=\B\cap\C$. 
\begin{proof}
Let $X$ be an object in $\B'\cap\C'$. Since $(\B',\B,\C,\C')$ is an adjoint sequence, the identity morphism $\iden_X$ factors as follows$\colon$
\[
\begin{tikzcd}
\B'\ni X \arrow[rrr, bend right, "\iden_X"]  \arrow[r, "\lambda_{X}"] & B \arrow[r, "\lambda_B"] & C \arrow[r, "\lambda_{C}"] & X \in \C' 
\end{tikzcd}
\]
The adjoint pair $(\C,\C')$ arises from the coreflective subcategory $\C\hra\C'$. This implies that the morphism $\lambda_C$ is an isomorphism and therefore $X$ is isomorphic to an object in $\C$. By symmetry, we obtain that $X$ lies in $\B\cap\C$, i.e. $\B'\cap \C'\subseteq \B\cap\C$. Similarly, we show that
$\B\cap \C\subseteq \B'\cap\C'$.
\end{proof}
\end{prop} 

Our aim is to understand via the notion of an adjoint pair the structure of  $\smod\PP$ and $\com\II$ as a pair of subcategories in the derived category $\mD(\A)$. To proceed we need the following preliminary result.

\begin{lem}
\label{lemfactorisationprop}
Let $\A$ be an Ext-category with $\mathscr{I}:=\Inj(\A)$ right coherent. Then any morphism $f\colon A\lxr B$ in $\A$ admits a 
morphism $g\colon B\lxr I$ with $I\in\II$ and $gf=0$ such that each morphism $g'\colon B\lxr I'$ with $I'\in\II$ and $g'f=0$ factors through $g$.
\begin{proof}
Let $f\colon A\lxr B$ be a morphism in $\A$ and choose an acyclic complex
$0\lxr A\stackrel{a}{\lxr} I_0\stackrel{i}{\lxr} I_1\lxr A'\lxr 0$ in $\A$ with $I_0, I_1$ in $\II$. Let $b\colon B\lxr I_2$ be an inflation with $I_2$ in $\II$. Then there is a morphism $j\colon I_0\lxr I_2$ such that $bf=ja$. Since $\II$ is right coherent, there exists a weak cokernel $(j'\; i')\colon I_1\oplus I_2\lxr I$ of $\binom{-i}{j}\colon I_0\lxr I_1\oplus I_2$ in $\II$. In particular, we have the following commutative diagram$\colon$
\[
\begin{tikzcd}
A \arrow[r, rightarrowtail, "a"] \arrow[d, "f"] & I_0 \arrow[r, "i"] \arrow[d, "j"] & I_1 \arrow[d,"j'"]  \\
B \arrow[r, rightarrowtail, "b"] & I_2 \arrow[r, "i'"] & I  
\end{tikzcd}
\]
We claim that the map $g:=i'b\colon B\lxr I$ meets the requirement. Clearly, $gf=0$. Let $g'\colon B\lxr I'$ with $I'\in\II$ such that $g'f=0$. We show that $g'$ factors through $g$. First, there is a map $b'\colon I_2\lxr I'$ such that $b'b=g'$. We write $i=k'k$ for the canonical factorisation of the map $i\colon I_0\tra K \rat I_1$ . Since $b'ja=0$, there is a morphism $l\colon K\lxr I'$ such that $lk=b'j$. Then we also get a morphism $l'\colon I_1\lxr I'$ such that $l'k'=l$. Then, since $(l'\; b')\binom{-i}{j}=0$, there is a morphism $c\colon I\lxr I'$ making the following diagram commutative$\colon$
\[
\begin{tikzcd}
I_0 \arrow[r, rightarrowtail, "\binom{-i}{j}"]  & I_1\oplus I_2 \arrow[r, "(j'\; i')"] \arrow[d, "(l'\; b')"] & I \arrow[dl, bend left, "c"]  \\
 & I' &   
\end{tikzcd}
\]
Then the morphism $c$ satisfies $cg=g'$ and this completes the proof.
\end{proof}
\end{lem}

\begin{prop}
\label{p15}
Let $\A$ be an Ext-category with $\mathscr{P}:=\Proj(\A)$ left coherent
and $\mathscr{I}:=\Inj(\A)$ right coherent. Then 
\[
\big(\!\smod(\PP), \com(\II)\big) 
\]
is an adjoint pair in the derived category $\mD(\A)$ which extends to an adjoint pair $(\mD^{\le 0}(\A),\mD^{\ge 0}(\A))$.
\end{prop} 
\begin{proof}
Let 
\[
\begin{tikzcd}
\cdots \arrow[r] & P_{-2} \arrow[r, "a_{-2}"] & P_{-1} \arrow[r, "a_{-1}"] & P_0 \arrow[r] & 0  
\end{tikzcd}
\]
be an object  $M$ in $\smod(\mathscr{P})\hra \mD(\A)$. Since $\mathscr{I}$ is right coherent, there are morphisms $P_0\stackrel{e}{\lra} I_0\ra I_1\ra I_2\ra\cdots$ such that for any object $I\in\II$ the following complex 
\[
\begin{tikzcd}[cramped, sep=scriptsize] 
\cdots \arrow[r] & \Hom_{\A}(I_1,I) \arrow[r] & \Hom_{\A}(I_0,I) \arrow[r] & \Hom_{\A}(P_0,I) \arrow[r] & \Hom_{\A}(P_{-1},I)  
\end{tikzcd}
\]
 is exact. Hence the complex 
\[
\begin{tikzcd}
0 \arrow[r] & I_0 \arrow[r] & I_{1} \arrow[r] & I_2 \arrow[r] & \cdots  
\end{tikzcd}
\]
is an object $C$ in $\com(\II)$. Moreover, the map $e\colon P_0\lxr I_0$ induces a morphism $\lambda_M\colon M\ra C$ in $\mD(\A)$ such that every morphism $M\lxr C'$ with $C'\in\com\II$ factors uniquely through $\lambda_M$. By symmetry, this shows that $(\smod(\PP),\com(\II))$ is an adjoint pair.

By Lemma~\ref{Keller'slemma} (i), every object in $\mD^{\le 0}(\A)$ is isomorphic to a complex of the form$\colon$$\cdots\ra P_{-1}\stackrel{a_{-1}}{\lra} P_0\ra 0$, which determines an object $\Coker{a_{-1}}\in\smod(\PP)$. This assignment gives a left adjoint $\mD^{\le 0}(\A)\lxr\smod(\PP)$ to the inclusion functor
$\smod(\PP)\hra \mD^{\le 0}(\A)$, i.e.\ $\smod(\PP)$ is a reflective subcategory of $\mD^{\le 0}(\A)$. Similarly, $\com(\II)$ is a coreflective full subcategory of $\mD^{\ge 0}(\A)$. Moreover, any morphism $P\lxr I$ in $\mD(\A)$ with $P$ in $\mD^{\le 0}(\A)$ and $I$ in $\mD^{\ge 0}(\A)$ factors uniquely through $\lambda_P\colon P\lxr P'$ with $P'$ in $\smod(\PP)$, and every $M\lxr I$ with $M$ in $\smod(\PP)$ and $I$ in $\mD^{\ge 0}(\A)$ factors uniquely through 
$\lambda_M\colon M\lxr C$ with $C$ in $\com(\II)$. Hence the tuple $(\mD^{\le 0}(\A), \smod(\PP), \com(\II), \mD^{\ge 0}(\A))$ is an adjoint sequence in $\mD(\A)$. 
\end{proof}

\begin{cor}
For an Ext-category $\A$ with $\PP:=\Proj(\A)$ left coherent and $\II:=\Inj(\A)$ right coherent, the following equivalence holds in 
$\mD(\A)\colon$
\[
\A\simeq \smod(\PP)\cap \com(\II)\simeq \mD^{\le 0}(\A)\cap \mD^{\ge 0}(\A).
\]      
\begin{proof}
By Proposition~\ref{p14}, it is enough to verify the first equivalence.
For an object $A$ in $\smod(\PP)\cap \com(\II)$ with projective
resolution $P$ and injective resolution $I$, there is a morphism 
$P\lxr I$ in $\mk(\A)$ with an acyclic mapping cone. Hence the object $A$ lies in $\A$.
\end{proof}
\end{cor}

The following corollary shows that the orthogonal pairs (\ref{11}) give rise to $t$-structures (see \cite[Lemma~6.3]{Fio}).

\begin{cor}
\label{cortstructure}
Let $\A$ be an Ext-category with $\PP:=\Proj(\A)$ left coherent and $\II:=\Inj(\A)$ right coherent. For any object $A\in \mD(\A)$ and $n\in\mathbb{Z}$, there is a triangle
\begin{equation}\label{14}
\begin{tikzcd}
\tau_{\le n} A \arrow[r] & A \arrow[r] & \tau^{>n}A \arrow[r] &  \tau_{\le n} A[1]  
\end{tikzcd}
\end{equation}  
in $\mD(\A)$ with $\tau_{\le n}A \in \mD_{\le n}(\A)$ and $\tau^{>n}A \in \mD^{>n}(\A)$.
\begin{proof}
There is an obvious triangle $A^{>0}\ra A\ra A^{\le 0}\stackrel{\delta}{\ra} 
A^{>0}[1]$ with $A^{>0}\in\mD^{>0}(\A)$ and
$A^{\le 0}\in\mD^{\le 0}(\A)$. So we have a universal morphism 
$\lambda_{A^{\le 0}}\colon A^{\le 0}\lxr C$ with $C\in\com(\II)$ according to the adjoint pair $(\mD^{\le 0}(\A), \com(\II))$. The mapping cone of $\lambda_{A^{\le 0}}$ is in $\mD_{<-1}(\A)$. Since $A^{>0}[1]$ lies in $\mD^{\ge 0}(\A)$, the morphism $\delta$ factors through $\lambda_{A^{\le 0}}$. So we get an octahedron
\[
\begin{tikzcd}
& \tau_{< 0} A \arrow[d] \arrow[r, equal] & \tau_{< 0} A \arrow[d] &  \\
A^{>0} \arrow[d, equal] \arrow[r] & A \arrow[r] \arrow[d] & A^{\le 0} \arrow[d] \arrow[r] & A^{>0}[1] \arrow[d, equal] \\
A^{>0} \arrow[r] & \tau^{\ge 0}A \arrow[r] & C \arrow[r] & A^{>0}[1] 
\end{tikzcd}
\]
with $\tau^{\ge 0} A\in \mD^{\ge 0}(\A)$ and $\tau_{<0} A\in \mD_{<0}(\A)$. This proves (\ref{14}) for $n=-1$, hence for all $n\in\mathbb{Z}$. 
\end{proof}
\end{cor}

Similarly, there is the following triangle in $\mD(\A)\colon$
\begin{equation}\label{15}
\begin{tikzcd}
\tau_{<n} A \arrow[r] & A \arrow[r] & \tau_{\ge n}A \arrow[r] &  \tau_{<n} A[1]  
\end{tikzcd}   
\end{equation}  
with $\tau_{< n}A \in \mD_{< n}(\A)$ and $\tau^{\ge n}A \in \mD^{\ge n}(\A)$ according to the second equation in (\ref{11}). For exact categories
with splitting idempotents and a coherence condition (which generalizes that
for $\PP$ and $\II$ in Proposition~\ref{p15}), the existence of these 
truncation functors was proved recently by Fiorot \cite{Fio}.  Note that by (\ref{11}), the triangles (\ref{14}) and (\ref{15}) are functorial. Hence the pairs $(\mD_{\le 0}(\A),\mD^{\ge 0}(\A))$ and $(\mD^{\le 0}(\A),\mD_{\ge 0}(\A))$ are $t$-structures in $\mD(\A)$.

In \cite{RumpFCMhdim} a notion of a dimension of an Ext-category $\A$ was defined. In a sense, this dimension measures how far is $\A$ from being abelian, i.e.\ equivalent to $\smod(\PP)$. This notion is necessary for our discussion on totally acyclic complexes over exact categories in Section~\ref{section:totallyacyclic}.

We now provide the desired definition of a dimension for $\A$. The reader is advised to compare it with \cite[Section 5, Definition~5]{RumpFCMhdim}. It is basically the same, we use again the full embeddings $\A\lxr \smod(\PP)\lxr \mD(\A)$ by Proposition~\ref{p13} and our notation for the subcategory of left exact complexes.

\begin{defn}
\label{d6}
Let $\A$ be an Ext-category with $\PP:=\Proj(\A)$ left 
coherent and $\II:=\Inj(\A)$ right coherent. We define the {\bf 
dimension} $\dime\A$ of $\A$ to be the infimum of all $n\in\mathbb{N}$ with 
$\mD_{\ge n}(\A)\subs \mD^{\ge 0}(\A)$. 
\end{defn} 

The dimension $\dime\A$ lies in $\mathbb{N}\cup\{\infty\}$. By Corollary~\ref{cortstructure}, $\mD_{\ge n}(\A)
\subs \mD^{\ge 0}(\A)$ is equivalent to $\mD_{\le 0}(\A)\subs
\mD^{\le n}(\A)$. Thus the Definition~\ref{d6} is left-right symmetric. 

One important aspect of the dimension $\dime\A$ of $\A$ is that $\dime\A<\infty$ implies that the category of projectives $\PP:=\Proj(\A)$ is left coherent, see \cite[Proposition~15]{RumpFCMhdim}. We close this section with the following result where we characterise the finiteness of $\dime\A$ in terms of the $\A$-resolution dimension of $\smod(\PP)$ in the sense of Auslander-Buchweitz, see \cite[Remark~2]{RumpFCMhdim}.

\begin{prop}
\label{p16}
Let $\A$ be an Ext-category with $\PP:=\Proj(\A)$ left 
coherent and $\II:=\Inj(\A)$ right coherent. For any $n\in\mathbb{N}$, the 
following are equivalent$\colon$
\begin{enumerate}
\item $\dime\A\le n$. 

\item Every exact sequence $0\lxr M_0\lxr P_1\lxr\cdots\lxr P_n\lxr M_n\lxr 0$ in 
$\smod(\PP)$ with $P_i\in\PP$ satisfies $M_0\cong A_0$ for some $A_0\in\A$. 

\item Every exact sequence $0\lxr M_0\lxr A_1\lxr\cdots\lxr A_n\lxr M_n\lxr 0$ in $\smod(\PP)$ with $A_i\in\A$ satisfies $M_0\cong A_0$ for some $A_0\in\A$.
\end{enumerate} 
\begin{proof}
(i)$\Longrightarrow$(ii)$\colon$Consider a projective 
resolution $P_1\lxr\cdots\lxr P_n$ of $M_n$ in $\smod\PP$. This gives a complex $P$ over $\PP$ with $P_i=0$ for $i>n$ such that $P$ is left exact at $m<n$. Thus $P$ belongs to $\mD^{\le n}(\A)\cap \mD_{\ge n}(\A)$ and since $\dime\A\le n$, i.e.\ $\mD_{\ge n}(\A)\subs \mD^{\ge 0}(\A)$, it follows that $P$ lies in $\mD^{\le n}(\A)\cap \mD^{\ge 0}(\A)$. This means that $P_i=0$ for all $i>n$ and all $i<0$ in $\mD(\A)$. 
Since $P_i$ belongs to $\PP$ for all $i$, there is a triangle $P\lxr A\lxr Z\lxr P[1]$ in $\mathsf{K}(\A)$ with $A\in \mathsf{K}^{\ge 0}(\A)$ and $Z$ acyclic, see Lemma~\ref{Keller'slemma}. In particular, since $P_i=0$ for all $i<0$ in $\mD(\A)$ this means that there is a factorization $P_{-1}\tra B\rat P_0$ in $\A$. Using the embedding $\A\lxr \smod(\PP)$ (see Proposition~\ref{p13}), the latter factorisation shows that the cokernel $M_0$ of $P_{-1}\lxr P_0$ is isomorphic to an object in $\A$.  

(ii)$\Longrightarrow$(iii)$\colon$Let $0\lxr M_0\lxr A_1\lxr\cdots\lxr A_n\lxr M_n\lxr 0$ be an exact sequence in $\smod(\PP)$ with $A_i\in\A$. We show that $M_0\cong A_0$ for some $A_0\in\A$. Denote by $A$ the complex $A_1\lxr\cdots\lxr A_n$. Then, by Lemma~\ref{Keller'slemma} (i) there is a triangle $P\lxr A\lxr Z\lxr P[1]$ in $\mathsf{K}(\A)$ with $A\in \mathsf{K}^{\ge 0}(\A)$ and $Z$ acyclic. Note that $P_i$ has projective terms and by the proof of Lemma~\ref{Keller'slemma} there exists a quasi-isomorphism $\rho\colon P\lxr A$. Thus, we get an exact sequence $0\lxr M_0'\lxr P_1\lxr\cdots\lxr P_n\lxr M_n'\lxr 0$ in 
$\smod\PP$ with $P_i\in\PP$ and by (ii) there is an object $A_0$ in $\A$ such that $M_0'\cong A_0$. Since the complexes $P$ and $A$ are identified in $\mD(\A)$, the desired property holds for $M_0$ as well.

(iii)$\Longrightarrow$(i)$\colon$This implications follows immediately by definition.
\end{proof}
\end{prop}

\begin{rem}
\label{remfindim}
We remark that $\dime\A=0$ if and only if $\A\simeq \smod\PP$, that is, $\A$ 
is abelian. Also, $\dim\A\le 1$ if and only if $\A$ is left semi-abelian \cite{qa}. Under the assumptions of Proposition~\ref{p16}, this is equivalent to $\A$ being quasi-abelian (\cite{Sch}; cf. \cite{Fio}, Section~1). Note that $\dime\A\le 2$ means that $\A$ has kernels (and cokernels, by symmetry).  For a Cohen-Macaulay order $\Lambda$ over a complete {$d$-dimensional} regular local ring, Proposition~\ref{p16} implies that $\dime \CM(\Lambda)\leq d<\infty$. 
\end{rem}

\section{Totally acyclic exact categories}
\label{section:totallyacyclic}

In this section we clarify the relationship between acyclicity with exactness.
Let $\A$ be an Ext-category with $\PP:=\Proj(\A)$ left coherent. As before, we denote by $\II$ the subcategory $\Inj(\A)$ of injective objects of $\A$. We consider two full subcategories of the abelian category $\MM:=\smod(\PP)$. The notation $\MM$ is used when we write $\Hom_{\MM}(M,N)$ or $\Ext^i_{\MM}(M,N)$ for two objects $M, N\in \smod(\PP)$. Define
an {\bf $\II$-resolution} of $M\in\MM$ to be an exact sequence 
\begin{equation}
\label{17}
\begin{tikzcd}
0 \arrow{r}  & M \arrow{r} & I_0 \arrow{r} & I_1 \arrow{r} & I_2 \arrow{r} & \cdots
\end{tikzcd}
\end{equation}
in $\MM$ with $I_n\in\II$ for all $n\in\mathbb{N}$ such that the sequence
\begin{equation}
\label{hominjexact}
\begin{tikzcd}
\cdots \arrow{r}  &  \Hom_{\MM}(I_1,I) \arrow{r} &  \Hom_{\MM}(I_0,I) \arrow{r} &  \Hom_{\MM}(M,I) \arrow{r} & 0 
\end{tikzcd}
\end{equation}
is exact for all $I$ in $\II$. We write $\mathsf{T}^-(\A)$ for the full subcategory of all $M\in\MM$ which admit an $\II$-resolution. By $\mt_-(\A)$ we denote the full subcategory of all $M\in\MM$ with $\Ext_{\MM}^j(M,I)=0$ for all $I\in\Inj(\A)$ and $j>0$. 

The intersection 
\begin{equation}
\label{18}
\mt(\A):=\mt^-(\A)\cap \mt_-(\A)  
\end{equation}  
will be called the {\bf acyclic closure} of $\A$.

\begin{prop}
\label{p17}
Let $\A$ be an Ext-category with $\PP$ left coherent. Then $\mt(\A)$ is the largest resolving, hence exact, subcategory $\C$ of $\smod(\PP)$ with
$\Proj(\C)=\add{\PP}$ and $\Inj(\C)=\add{\II}$. In particular, $\mt(\A)$ is an Ext-category.
\begin{proof}
Let $\C$ be a resolving subcategory of $\smod(\PP)$ which satisfies 
$\Proj(\C)=\add{\PP}$ and $\Inj(\C)=\add{\II}$, and let $M$ be an
object in $\C$. We show that $M$ lies in $\mt(\A)$.
Consider a conflation in $\C$ 
\[
\begin{tikzcd}
M'  \arrow[r, rightarrowtail, "i"] & P \arrow[r, twoheadrightarrow] & M  
\end{tikzcd}
\]
with $P$ in $\PP$. Since $\Inj(\C)=\add{\II}$, every morphism $M'\lxr I$ with $I$ in $\II$ factors through $i$. From the long exact sequence
\[
\begin{tikzcd}
\Hom_{\MM}(M,I) \arrow[r, rightarrowtail] & \Hom_{\MM}(P,I) \arrow[r, twoheadrightarrow] & \Hom_{\MM}(M',I)  \arrow[r] & \Ext^1_{\MM}(M,I) \arrow[r] & 0  
\end{tikzcd}
\]
it follows that $\Ext^1_{\M}(M,I)=0$. By dimension shift we infer that $M\in\mt_-(\A)$. 
Consider now a sequence of conflations in $\C\colon$ 
\[
\begin{tikzcd}
M  \arrow[r, rightarrowtail] & I_0 \arrow[r, twoheadrightarrow] & \Sigma(M)  
\end{tikzcd},  \begin{tikzcd}
\Sigma(M)  \arrow[r, rightarrowtail] & I_1 \arrow[r, twoheadrightarrow] & \Sigma^2(M)  
\end{tikzcd}, \ \ldots
\]
with $I_i$ in $\II$. Let us explain how we get these conflations. If you think $M$ as a two term complex $P_1\lxr P_2$, there is an inflation $P_1\to I_1$ in $\A$ with $I_1\in \II$. Then take the pushout diagram and consider again an inflation of the pushout as before. Recall that by the definition of an exact category, the pushout of an inflation along an arbitrary morphism exists and yields an inflation. Then from the embedding $\A\hra\smod(\PP)$ (see Proposition~\ref{p13}) we get the desired conflations in $\C$ since it is a resolving subcategory of $\smod(\PP)$ with $\Inj(\C)=\add{\II}$. From the exact structure of $\C$, we obtain an exact sequence of the form $(\ref{17})$ in $\smod(\PP)$. For any injective object $I$ in $\A$, applying the functor $\Hom_{\MM}(-,I)$ to $(\ref{17})$ we get 
the exact sequence $(\ref{hominjexact})$ since there is an isomorphism $\Ext^1_{\MM}(\Sigma^i(M),I)\cong \Ext^1_{\C}(\Sigma^i(M),I)=0$. Thus, $M$ belongs to $\mt^{-}(\A)$ and therefore $M$ lies in $\mt(\A)$.

By Proposition~\ref{p13}, it remains to show that 
$\mt(\A)$ is resolving. Clearly, the projectives $\PP$ lie in $\mt(\A)$. Consider a short exact sequence in $\smod(\PP)\colon$
\[
\begin{tikzcd}
0\arrow[r] & L \arrow[r, "a"] & M  \arrow[r, "b"] & N \arrow[r] & 0  
\end{tikzcd}
\]
with $L$ and $N$ in $\mt(\A)$. Then $M$ belongs to $\mt_-(\A)$.
Furthermore, there are $\II$-resolutions 
\[
\begin{tikzcd}
0\arrow[r] & L \arrow[r, "i"] & I_0  \arrow[r] & \cdots   
\end{tikzcd} \ \ \text{and} \ \ \begin{tikzcd}
0\arrow[r] & N \arrow[r, "j"] & J_0  \arrow[r] & \cdots   
\end{tikzcd}
\]
Since $N\in\mt_-(\A)$, we have $i=ha$ for some $h\colon M\lxr I_0$. Thus $\binom{h}{jb}\colon M\lxr I_0\oplus J_0$ is monic, and it is easily checked that every morphism $M\lxr I$ with $I\in\II$ factors through 
$\binom{h}{jb}$. Applying the $3\times 3$ lemma and induction, we get
$M\in\mt^-(\A)$. 

Next assume that $M$ and $N$ lie in $\mt(\A)$. Then clearly $L$ is in $\mt_-(\A)$. Furthermore, there is a conflation 
$M\rat I_0\tra M'$ in $\mt(\A)$ with $I_0\in\II$. This gives a commutative diagram
\[
\begin{tikzcd}
L \arrow[r, rightarrowtail, "a"] \arrow[d, equal] & M \arrow[r, twoheadrightarrow, "b"] \arrow[d, rightarrowtail, "i"] & N \arrow[d, rightarrowtail] \\
L \arrow[r, rightarrowtail, "ia"] & I_0 \arrow[d, twoheadrightarrow] \arrow[r, twoheadrightarrow] & L' \arrow[d, twoheadrightarrow] \\
 & M' \arrow[r, equal] & M'
\end{tikzcd}
\]
where every morphism $L\lxr I$ with $I\in\II$ factors through $a$, hence also through $ia$. Since $\mt(\A)$ is closed under extensions in $\smod(\PP)$ and $N$, $M'$ lie in $\mt(\A)$, we have $L'\in\mt(\A)$. 
Then the middle conflation in the diagram above is the start of the desired $\II$-resolution for $L$. We infer that $L$ belongs to $\mt(\A)$. 
\end{proof}
\end{prop}

As a consequence, we get an abstract description of $\mt(\A)$, 
removing the asymmetry caused by the embedding into $\smod(\PP)$. We say that an additive category $\A$ is a {\em variety} ({\em of annuli}) \cite{AR} if idempotents split in $\A$. By $\add\A:=\Proj(\smod(\A))\simeq \Inj(\com(\A))$ we denote the variety generated by $\A$. By ``largest'' we mean that it contains any other Ext-category which contains $\A$ as a full subcategory and has the prescribed projectives and injectives.

\begin{cor}
\label{corTTA} 
Let $\A$ be an Ext-category with $\PP:=\Proj(\A)$ left coherent and $\II:=\Inj(\A)$. Up to equivalence, $\mt(\A)$ is the largest Ext-category $\C$ containing $\A$ as a full subcategory such that $\Proj(\C)=\add\PP$
and $\Inj(\C)=\add\II$. In particular, 
$\mt(\mt(\A))=\mt(\A)$.  
\begin{proof}
By Proposition~\ref{p13}, any such category $\C$ admits a resolving full embedding 
$\C\hra\smod(\PP)$ such that $\C$ carries the induced exact structure from 
$\smod(\PP)$. Hence $\A\hra\C\hra\mt(\A)\hra\smod(\PP)$ with the 
exact structures induced from $\smod(\PP)$. 
\end{proof}
\end{cor}

We remark that the left coherence of $\PP$ can be dropped. To see this, one has to show that the {\em left abelian} \cite{T} category $\smod(\PP)$ still has a natural exact structure. 

The following result shows that the acyclic closure $\mathsf{T}(\A)$ of $\A$ consists of the Gorenstein-projectives in $\smod(\PP)$ (see subsection~\ref{subsectiongorprojectives}).

\begin{prop}
\label{p18}
Let $\A$ be an Ext-category with $\PP:=\Proj\A$ left coherent. 
An object $M\in\smod(\PP)$ belongs to $\mt(\A)$ if and only if 
there is an exact complex $(\ref{2})$ over $\A$ where $M$ is the image of 
$a_0$ in $\smod(\PP)$.
\begin{proof}
Assume first that $M$ lies in $\mt(\A)$. Since $\PP$ is left coherent, i.e.\ $\smod(\PP)$ is abelian, there exists a projective resolution $\cdots\lxr P_1\lxr P_0\lxr M\lxr 0$ of $M$ in $\smod(\PP)$. Combined with an $\II$-resolution (\ref{17}), this cleary gives a complex (\ref{2}) over $\A$ where $M$ is the image of $a_0$ in $\smod(\PP)$. Let us explain now why the latter complex is exact in the sense of Definition~\ref{d3}. We have constructed the following exact sequence in $\smod(\PP)\colon$
\begin{equation}
\label{resolexactcomplex}
\begin{tikzcd}
\cdots\arrow[r] & P_1 \arrow[r] & P_0 \arrow[rd, twoheadrightarrow]  \arrow[rr] && I_0 \arrow[r] & I_1\arrow[r] & \cdots   \\
 &  & &  M \arrow[ur, rightarrowtail, "i"'] &    
\end{tikzcd}
\end{equation}
First, the preimage of the above complex in $\A$ is left exact. This follows easily by the Yoneda embedding using also that the above sequence is exact in $\smod(\PP)$. We now show that the above complex is right exact in $\A$. Let $f\colon I_0\lxr I$ be a morphism in $\A$ such that $f\circ a_0=0$. Note that $a_0$ denotes the map $P_0\lxr I_0$ in $\A$. Passing now to $\smod(\PP)$ the latter composition being zero, implies that $f\circ i=0$. Since the right hand side of $(\ref{resolexactcomplex})$ is an $\II$-resolution, it follows that we have the exact sequence $(\ref{hominjexact})$. Thus, there is a morphism $g\colon I_1\lxr I$ such that $g\circ a_1=f$. Taking the preimage of this in $\A$, we infer that $(\ref{resolexactcomplex})$ is left exact at $n=1$. The same argument shows that $(\ref{resolexactcomplex})$ is left exact at $n=0$ and  similarly we show that $(\ref{resolexactcomplex})$ is left exact for all $n\geq 0$. Finally, the desired factorisation property for all $n\leq -1$, i.e. $(\ref{resolexactcomplex})$ is left exact for all $n\leq -1$, 
follows easily since $\Ext_{\MM}^j(M,I)=0$ for all $I$ in $\Inj(\A)$ and $j>0$. 

Conversely, let (\ref{2}) be an exact complex over $\A$ and let $M_n$ be the image of $a_n$ in $\MM:=\smod(\PP)$. By Proposition~\ref{p13}, the category $\A$ is a full resolving subcategory of $\smod(\PP)$. This implies that $(*)\colon\Ext^i_\A(A,A')\cong \Ext_{\MM}^i(A,A')$ for all $i\geq 0$ and all objects $A, A'$ in $\A$. Consider the following exact sequence in $\smod(\PP)$:
\begin{equation}
\label{exactinmodP}
\begin{tikzcd}
\cdots\arrow[r] & A_{-1} \arrow[d, twoheadrightarrow]  \arrow[r] & A_0 \arrow[rd, twoheadrightarrow]  \arrow[rr] && A_1 \arrow[r] & A_2\arrow[r] & \cdots   
\\
 &  M_{-1} \arrow[ur, rightarrowtail] &  &  M_0=:M \arrow[ur, rightarrowtail] &    
\end{tikzcd}
\end{equation}
Let $I$ be an object in $\II$. Then we have the following long exact sequence:
\[
\begin{tikzcd}
0 \rar &  \Hom_{\MM}(M,I) \rar & \Hom_{\MM}(A_0,I) \rar & \Hom_{\MM}(M_{-1},I) \ar[out=-30, in=150]{dll} \\
 & \Ext^1_{\MM}(M,I) \rar & \Ext^1_{\MM}(A_0,I) \rar & \Ext^1_{\MM}(M_{-1},I) \ar[out=-30, in=150]{dll} \\
 & \Ext^2_{\MM}(M,I) \rar & \Ext^2_{\MM}(A_0,I) \rar & \cdots 
\end{tikzcd}
\]
Since $(\ref{exactinmodP})$ is right exact, the map $\Hom_{\MM}(A_0,I)\lxr \Hom_{\MM}(M_{-1},I)$ is surjective. Moreover, by the isomorphism $(*)$ above we get that $\Ext^2_{\MM}(A_0,I)=0=\Ext^1_{\MM}(A_0,I)$. We infer that $\Ext^j_{\MM}(M,I)=0$ for all $j>0$, i.e.\ $M$ lies in $\mt_{-}(\A)$. By the proof of the dual of Lemma~\ref{Keller'slemma}, there exists a quasi-isomorphism $\rho\colon A\lxr I$ with $\rho_n=1_{A_n}$ for $n\le 0$ and $I_n\in\II$ for $n>0$. Using the full embeddings 
$\A \lxr \smod(\PP) \lxr \mD(\A)$ by Proposition~\ref{p13} and since $\A$ is right exact, we get that $0\lxr M\lxr I_1\lxr I_2\lxr\cdots$ is an $\II$-resolution. Hence, $M$ lies in $\mt^-(\A)$. 
\end{proof}
\end{prop} 

\begin{exam}
\label{examgorproj}
Let $R$ be a ring, and let $\Add(R)$ (respectively, $\add(R)$) be the category of (finitely generated) projective $R$-modules with the trivial exact structure. By Proposition~\ref{p18}, an $R$-module is Gorenstein-projective (see subsection~\ref{subsectiongorprojectives}) if and only if it belongs to the full subcategory $\mt(\Add(R))$ of $\Mod(R)$.  
\end{exam}

As an immediate consequence of Propositions~\ref{p17} and \ref{p18}, we have

\begin{cor}
\label{corexactacyclic}
Let $\A$ be an Ext-category with $\Proj(\A)$ left coherent. A complex over $\A$ is exact if and only if it is acyclic over $\mt(\A)$. 
\end{cor}

If $\II$ is right coherent, we can analogously define a pair of 
full subcategories $\mt^+(\A)$ and $\mt_+(\A)$ of 
$\com\II$ and consider their intersection $\mt'(\A)$ in 
$\com\II$. In $\mD(\A)$, the subcategories 
$\mt(\A)$ and $\mt'(\A)$ are equivalent (not necessarily
equal):

\begin{cor} 
Let $\A$ be an Ext-category with $\PP$ left coherent and $\II$ right 
coherent. Then $\mt(\A)$ consists of the objects $A\in\smod(\PP)$ for which the complex $\tau_{<0} A$ is exact. Furthermore, $\mt(\A)\simeq \mt'(\A)$.    
\end{cor}
\begin{proof}
The first statement follows by equation (\ref{18}). The equivalence between 
$\mt(\A)$ and $\mt'(\A)$ is given by 
$M\mapsto (\tau_{<0}M)^{>0}$ via the exact complexes $\cdots\ra P_{-1}\ra P_0\ra I_1\ra I_2 \ra\cdots$ with $P_{-n}\in\PP$ and $I_{n+1}\in\II$ for $n\in\mathbb{N}$. 
\end{proof}

\begin{cor}
Let $\A$ be an Ext-category with $\Proj(\A)$ left coherent. Then $\mt(\A)$ is a variety.
\begin{proof}
Let $e\colon M\lxr M$ be an idempotent in $\mt(\A)$. Then 
\[
\begin{tikzcd}
\cdots \arrow{r}  & M \arrow[r, "e"] & M \arrow{r}{1-e} & M \arrow[r, "e"] & M \arrow{r}{1-e} & M \arrow{r} & \cdots
\end{tikzcd}
\]
is an exact complex over $\mt(\A)$, hence acyclic by Corollary~\ref{corexactacyclic}. Thus $e$ splits.
\end{proof}
\end{cor}

\begin{defn}
\label{d7}
Let $\A$ be an Ext-category with $\PP:=\Proj(\A)$ left 
coherent and $\II:=\Inj(\A)$ right coherent. We call $\A$ {\bf totally
acyclic} if every exact complex over $\A$ is acyclic, that is, 
$\mt(\A)=\A$. 
\end{defn} 

There is an important special case, given by the following

\begin{prop}
\label{p19}
Let $\A$ be an Ext-category with $\PP$ left coherent and 
$\II$ right coherent. If $\dim\A<\infty$, then $\A$ is 
totally acyclic, with
\[
\A = \mt^-(\A) = \mt_-(\A) = \mt^+(\A)=
\mt_+(\A).
\]   
\end{prop} 
\begin{proof}
By symmetry, it suffices to verify first two equations. The first one
follows by Proposition~\ref{p16}. Now assume that 
$M\in\smod(\PP)$ belongs to $\mt_-(\A)$. Consider a projective 
resolution $\cdots\ra P_{n-1}\stackrel{d_{n-1}}{\lra} P_{n-2}\ra\cdots\ra P_0\ra M\ra 0$ in $\smod(\PP)$ with $n:=\dim \A$. Then $A:=\Ker{d_{n-1}}\in\A$. So there is an inflation $i\colon A\rat I$ with $I\in\II$. If $n>0$, then $i$ factors through $A\ra P_{n-1}$. Hence $A\ra P_{n-1}$ is an inflation. So the image $A'$ of $d_{n-1}$ belongs to $\A$, and we can repeat the same argument for $A'$ instead of $A$. By induction, this yields $M\in\A$. Whence $\mt_-(\A)=\A$.     
\end{proof}

\begin{rem}
We close this section with some remarks on $\mt(\A)$.
\begin{enumerate}
\item The main challenge of this work was to define ``big-objects'' over an arbitrary exact category $\A$. The category $\mt(\A)$ will play a key role towards this problem but it turns out that it is ``too big'' for our purposes. This will become clear in the next section, see Definition~\ref{d8} and Remark~\ref{d8remark}.     

\item Let use rephrase Example~\ref{examgorproj}. Assume that $\A$ is a skeletally small additive category with the trivial exact structure. Then the acyclic closure $\mt(\A)$, as an exact category, is equivalent to the category $\GProj(\Mod(\A))$ of Gorenstein projective $\A$-modules.

\item There is a more abstract approach to $\mt(\A)$ by the second author \cite{RumpAcyclicclosure}. It turns out that $\mt(\A)$ is strongly connected to (algebraic) triangulated categories, in particular, the quotient $\mt(\A)/\A$ admits a triangulated structure. We refer to \cite{RumpAcyclicclosure} for more details and further investigations of $\mt(\A)$ with respect to tilting theory.
\end{enumerate}
\end{rem}

\section{Big Cohen-Macaulay modules}
\label{sectionAuslandertyperesult}
For a commutative noetherian local algebra $R$ over a field, the existence of big Cohen-Macaulay modules was proved by Hochster \cite{Hoc}. In his survey \cite{Ho79}, he lists nine homological conjectures which follow by this existence theorem (see \cite{PS, Ho73, Ho74} for earlier results). Griffith \cite{Gri0} refined Hochster's theorem by showing that over a complete regular local ring $R$, any module-finite domain $S$ with a big Cohen-Macaulay module admits a countably generated one. In \cite{Gri} he dealt with the question when a countably generated big Cohen-Macaulay module over complete local Gorenstein ring $R$ splits into a direct sum of finitely generated ones. He proved this if $R$ is representation-finite (\cite[Corollary 5.2]{Gri}). Recently, Bahlekeh-Fotouhi-Salarian \cite{Bahlekeh} proved that balanced big Cohen-Macaulay modules with an $m$-primary cohomological annihilator are fully decomposable provided that there is a bound on the h-length.

For artinian algebras, a similar result was obtained much earlier by Ringel and Tachi\-ka\-wa \cite{Tach, RT}. Auslander \cite{Au76} established the converse, so that an artinian algebra $A$ is represen\-ta\-tion-finite if and only if every $A$-module is a direct sum of finitely generated ones. In \cite{Ind}, the second author extend this theorem to Krull dimension 1, that is, classical orders over a complete discrete valuation domain. More recently, Beligiannis \cite{Bel} obtained a similar decomposition theorem for Gorenstein projectives over a complete noetherian commutative local ring $A$, provided that there exists a non-projective finitely generated Gorenstein projective $A$-module. For artinian algebras $A$, he proved that Gorenstein projectives split into finitely generated modules if and only if $A$ is virtually Gorenstein and CM-finite, improving Chen's theorem \cite{Chen} for Gorenstein algebras $A$.

In this section, we extend Auslander's theorem to a very general class of
exact categories $\A$, so that all these improvements follow by 
specializing $\A$ to categories of finitely generated Cohen-Macaulay modules or Gorenstein-projectives, respectively. In particular, we obtain a necessary and sufficient {\em splitting-big-objects} criterion for Cohen-Macaulay orders of arbitrary finite dimension.

As a first step, we give a general definition of a``big'' object. Let $\A$ be a skeletally small left Ext-category (see Section~\ref{sectionprelim}) with $\PP:=\Proj(\A)$. For a full subcategory $\C$ of $\Mod(\PP)$, where $\Mod(\PP)$ denotes the category of $\PP$-modules (i.e.\ additive functors from $\PP^{\op}$ to the category of abelian groups), we define $\Add \C$ to be the full subcategory of direct summands of coproducts $\coprod_{\gamma\in\Gamma} C_\gamma$ with $C_\gamma\in\C$. Since every object $A\in\A$ admits a deflation $P\tra A$ with $P\in\PP$, the objects of $\A$ are {\em compact} in $\Add\A$, that is, any morphism $A\lxr \coprod_{\gamma\in\Gamma} A_\gamma$ with $A,A_\gamma\in\A$ factors through a finite subcoproduct of $\coprod_{\gamma\in\Gamma} A_\gamma$. By \cite[Section~1]{Lat}, this implies that $\Add\A$ is
equivalent to ${\bf Add}\A:=\Proj(\Mod(\A))$, a category 
which can be constructed from $\A$ via formal coproducts. 

To obtain big objects except those in ${\bf Add}\A$, we define an 
increasing sequence $\ke_0(\A)\subs\ke_1(\A)\subs\ke_2(\A)\subs\cdots$ of full subcategories of $\Mod(\PP)$ as 
follows. First, we set $\ke_0(\A):=\Add\A$. Inductively, we
define $\ke_{n+1}(\A)$ to be the full subcategory of objects $L\in
\Mod(\PP)$ which fit into a short exact sequence $0\lxr L\lxr M\lxr A\lxr 0$ 
with $M\in\ke_n(\A)$ and $A\in\Add\A$. The union of all $\ke_n(\A)$ will be denoted by $\ke(\A)$. This subcategory will play a crucial role in the sequel and its objects will be some of the ``big objects'' of $\A$, the constructive ones (see Definition~\ref{d8}).

\begin{prop}
\label{p20}
Let $\A$ be a skeletally small Ext-category with $\PP:=\Proj(\A)$ and $\II:=\Inj(\A)$. Then $\ke(\A)$ is a resolving subcategory of $\Mod(\PP)$, and an Ext-category with $\Proj(\ke(\A))=\Add\PP$ and $\Inj(\ke(\A))=\Add\II$. 
\end{prop}
\begin{proof}
We show first that $\Ext^1_{\Mod(\PP)}(M,I)=0$ for all $M$ in $\ke(\A)$ and $I$ in $\Add\II$. If $M$ is in $\Add\A$, then the assertion follows by the compactness of objects in $\A$. We now explain this, since is the key step of the induction that follows. Let $A'\lxr P\lxr A$ be a conflation in $\A$ with $P$ in $\PP$ and consider a map $A'\lxr \oplus_{i\in I}I_i$ with $I_i$ in $\II$. Since the objects of $\A$ are compact in $\Add\A$, the map $A'\lxr \oplus_{i\in I}I_i$ factors through a finite coproduct, i.e. through a map $A'\lxr \oplus_{j\in J}I_i$ with $|J|<\infty$. Taking the pushout of $A'\lxr P\lxr A$ along the map $A'\lxr \oplus_{j\in J}I_i$  we get a split conflation $\oplus_{j\in J}I_i\lxr P'\lxr A$ in $\A$ and thus a split sequence in $\Mod(\PP)$. This implies that also the induced conflation  $\oplus_{i\in I}I_i\lxr P''\lxr A$ splits. This shows that $\Ext^1_{\Mod(\PP)}(A,I)=0$.
Thus, by induction, assume that there is a short exact sequence $0\lxr M\lxr M'\lxr A\lxr 0$ with $M'$ in $\ke(\A)$ and $A$ in $\Add\A$ such that the assertion holds for $M'$ instead of $M$. Then the exact sequence $\Ext^1_{\Mod(\PP)}(M',I)\lxr \Ext^1_{\Mod(\PP)}(M,I)\lxr \Ext^2_{\Mod(\PP)}(A,I)$ proves the claim.

To show that $\ke(\A)$ is closed under extensions in $\Mod(\PP)$, let $L\stackrel{a}{\rat} M\stackrel{b}{\tra} N$ be a short exact sequence in $\Mod(\PP)$ with $L\in\ke_n(\A)$ and $N\in\ke(\A)$. Assume first that $n>0$. So there is a short exact sequence $L\rat L'\tra A$ with $L'\in\ke_{n-1}(\A)$ and $A\in \Add\A$. This gives a commutative diagram 
\begin{equation}\label{19}
\begin{tikzcd}
L \arrow[r, rightarrowtail, "a"] \arrow[d, rightarrowtail, "i"] & M \arrow[r, twoheadrightarrow, "b"] \arrow[d, rightarrowtail] & N \arrow[d, equal] \\
L' \arrow[r, rightarrowtail] \arrow[d, twoheadrightarrow] & M' \arrow[d, twoheadrightarrow] \arrow[r, twoheadrightarrow] & N  \\
A  \arrow[r, equal]  & A &
\end{tikzcd}  
\end{equation} 
with short exact rows and columns. If we can show that $M'\in\ke(\A)$, then $M\in\ke(\A)$. Hence, by induction, we can assume that $n=0$. So there is a
short exact sequence $L\rat L'\tra A$
with $L'\in\Add\II$ and $A\in\Add\A$. This gives a commutative diagram (\ref{19}) where $i$ factors through $a$. Hence $M'\cong L'\oplus N\in\ke(\A)$, and thus $M$ lies in $\ke(\A)$. 

So the short exact sequences in $\Mod(\PP)$ induce an exact structure on 
$\ke(\A)$. Now we show that $\Add\PP$ provides enough 
projectives for this exact structure. Let $L\in\ke_n(\A)$ be given. 
If $n=0$ there is a conflation $L'\rat P\tra L$ with $P\in \Add\PP$ and $L'\in\Add\A$. Thus, assume that $n>0$. Then we
have a conflation $L\rat M\tra A$ 
with $M\in\ke_{n-1}(\A)$ and $A\in\Add\A$. By induction, we
can assume that there exists a conflation $M'\rat P\tra M$ in 
$\ke(\A)$ with $P\in \Add\PP$. This gives a commutative diagram 
\[
\begin{tikzcd}
M' \arrow[r, equal] \arrow[d, rightarrowtail] & M' \arrow[d, rightarrowtail] &  \\
A' \arrow[r, rightarrowtail] \arrow[d, twoheadrightarrow] & P \arrow[d, twoheadrightarrow] \arrow[r, twoheadrightarrow] & A \arrow[d, equal]  \\
L  \arrow[r, rightarrowtail]  & M \arrow[r, twoheadrightarrow] & A
\end{tikzcd} 
\]
with exact rows and columns. Furthermore, we have a conflation $A''\rat Q\tra A$ with $Q\in\Add\PP$ and $A''\in \Add\A$. By Schanuel's lemma, $A'\oplus Q\cong A''\oplus P$. Hence $A'\in\Add\A$. So there is a deflation $P'\tra A'\tra L$ in $\ke(\A)$ with $P'\in\Add\PP$. This proves that 
$\ke(\A)$ has enough projectives. Since any deflation onto a projective object of $\ke(\A)$ splits, we obtain $\Proj(\ke(\A))=\Add\PP$.

Next, let $L\rat M\tra N$ be a short exact sequence in $\Mod(\PP)$ 
with $M,N\in\ke(\A)$. Then there is a conflation $N'\rat P\tra N$ in 
$\ke(\A)$ with $P\in\Add\PP$. So we have a commutative 
diagram 
\[
\begin{tikzcd}
N' \arrow[d] \arrow[r, rightarrowtail] & P \arrow[d] \arrow[r, twoheadrightarrow] & N \arrow[d, equal]  \\
L \arrow[r, rightarrowtail] & M \arrow[r, twoheadrightarrow] & N  
\end{tikzcd} 
\]
Hence $N'\rat L\oplus P\tra M$ is a conflation in $\ke(\A)$. Thus 
$L\rat L\oplus P\tra P$ is a short exact sequence, which yields that $L$ belongs to $\ke(\A)$. This proves that $\ke(\A)\hra\Mod(\PP)$ is resolving. 

By construction, every object $M\in\ke(\A)$ admits a finite 
sequence of inflations 
\[
\begin{tikzcd}
M \arrow[r, rightarrowtail] & M_1 \arrow[r, rightarrowtail] & M_2 \arrow[r, rightarrowtail] & \cdots \arrow[r, rightarrowtail] & A \arrow[r, rightarrowtail] & I  
\end{tikzcd} 
\]
with $A\in\Add\A$ and $I\in\Add\II$, and any object of $\II$ is injective in $\ke(\A)$. We infer that $\Inj(\ke(\A))=\Add\II$.
\end{proof}

\begin{rem}
\label{rembigobjects}
What we really show in the above result is that if
$\A$ is a skeletally small left Ext-category, then $\ke(\A)$ is a resolving subcategory of $\Mod(\PP)$ with $\Proj(\ke(\A))=\Add\PP$. Moreover, if
$\A$ is an Ext-category, then $\ke(\A)$ is an 
Ext-category with $\Inj(\ke(\A))=\Add\II$.

We infer that $\ke(\A)$ is the smallest resolving subcategory of $\Mod(\PP)$ 
which contains $\Add\A$. The fact that $\ke(\A)$ is the smallest follows by the construction of $\ke(\A)$. More precisely, suppose that $\C$ is a resolving subcategory of $\Mod(\PP)$ with the projectives and injectives of $\ke(\A)$ and which contains $\Add\A$. We claim that $\ke(\A)\subseteq \C$. Take an object $L$ in $\ke(\A)$ and suppose first that it lies in  $\ke_0(\A)$. Then $L$ lies in $\Add\A$, so it lies in $\C$. If $L$ belongs to $\ke_1(\A)$, then there is an exact sequence $0\lxr L\lxr M\lxr A\lxr 0$ 
with $M\in\Add\A$ and $A\in\Add\A$. Since $\C$ is resolving, we infer that $L\in \C$. In the same way, $L$ in $\ke(\A)$ means that $L$ belongs to $\ke_n(\A)$ for some $n$, and as before we obtain that $L\in \C$. This shows our claim and therefore $\ke(\A)$ is indeed the smallest resolving subcategory of $\Mod(\PP)$ which contains $\Add\A$.

Since $\Add\PP$ is left coherent, Proposition~\ref{p13} implies that up to equivalence, $\ke(\A)$ is 
the smallest Ext-category $\C$ containing ${\bf Add}\A$ with 
$\Proj\C\simeq {\bf Add}\PP (:=\Proj(\Mod(\PP)))$. So it must be contained in any
reasonable category of ``big'' objects with respect to $\A$. 
\end{rem}
   
It has been observed that well-behaved ``big objects'' should be 
representable as filtered colimits of small ones. For example, the Gorenstein projective modules over a CM-finite artinian algebra do not decompose into finitely generated ones unless they are direct limits of them (\cite[Theorem~4.10]{Bel}). We briefly discuss the role of this condition for an Ext-category $\A$ with $\PP:=\Proj(\A)$. Let $\varinjlim\A$ be the full subcategory of objects in $\Mod(\PP)$ which are filtered colimits of objects in $\A$. The following criterion is based on a well-known fact about direct limits (cf. \cite{Len}, Proposition~2.1). For completeness, we include a proof.

\begin{lem}
\label{l5}
Let $\A$ be a skeletally small Ext-category with $\PP:=\Proj(\A)$.
An object $M\in\Mod(\PP)$ belongs to $\varinjlim\A$ if and only if any
morphism $E\lxr M$ with $E\in\smod(\PP)$ factors through an object of $\A$.      
\begin{proof}
($\Longrightarrow$) Let $M=\varinjlim_{\gamma\in\Gamma} A_\gamma$ be a filtered colimit with 
$A_\gamma\in\A$, and let $f\colon E\lxr M$ be a morphism with $E\in\smod(\PP)$. Choose a presentation $P_1\stackrel{p}{\lra} P_0\stackrel{q}{\tra} E$ with $p\in\PP$. Then $fq=a_\gamma g$ for some $a_\gamma\colon A_\gamma\lxr M$ and $g\colon P_0\lxr A_\gamma$. Hence $a_\gamma gp=0$. So there exists a morphism $a_{\gamma,\delta}\colon A_\gamma\lxr A_\delta$ in the limit diagram which satisfies $a_{\gamma,\delta}gp=0$. Hence $a_{\gamma,\delta}g=hq$ for some $h\colon E\lxr A_\delta$, and thus $f=a_\delta h$. 

($\Longleftarrow$) Conversely, assume that the criterion holds for $M\in\Mod(\PP)$. Consider the diagram of all morphisms $a\colon A\lxr M$ with $A\in\A$. Morphisms from $a$ to $a'\colon A'\ra M$ are the morphisms $f\colon A\lxr A'$ in $\A$ with $a=a'f$. Then $\PP\subs\A$ implies that $M$ is the colimit of its
diagram. Any pair $a\colon A\lxr M$ and $a'\colon A'\lxr M$ is majorized by
$(a\;a')\colon A\oplus A'\lxr M$. For any pair $f,g\colon a\lxr a'$ of morphisms in the diagram we have $a'(f-g)=0$. Thus $a'$ factors through $\Cok(f-g)\in\smod(\PP)$, hence through an object of $\A$. So the diagram is filtered, which proves that $M\in\varinjlim\A$.
\end{proof}
\end{lem}

\begin{prop}
\label{proplimresolving}
Let $\A$ be a skeletally small left Ext-category with $\PP$ left coherent. Then $\varinjlim\A$ is a resolving subcategory of $\Mod(\PP)$ with $\ke(\A)\subset \varinjlim\A$.
\begin{proof}
Let $L\stackrel{a}{\rat} M\stackrel{b}{\tra} N$ be a short exact sequence 
in $\Mod(\PP)$. Assume first that $L,N\in\varinjlim\A$. Let $f\colon E\lxr M$ be a morphism with $E\in\smod(\PP)$. So there are morphisms $e\colon E\lxr A$
and $g\colon A\lxr N$ with $A\in\A$ and $ge=bf$. Choose a deflation $p\colon P\tra A$ with $P\in\PP$. Then there is a morphism $g'\colon P\ra M$ with $gp=bg'$. Consider the following pullback diagram in $\smod(\PP)\colon$
\[
\begin{tikzcd}
E' \arrow[d, "e'"] \arrow[r, twoheadrightarrow, "q"] & E \arrow[d, "e"] \\
P \arrow[r, twoheadrightarrow, "p"] & A   
\end{tikzcd} 
\]
Since $b(g'e'-fq)=0$, we find a morphism $h\colon E'\lxr L$ with $g'e'-fq=ah$. So there are morphisms $p'\colon E'\lxr A'$ and 
$a'\colon A'\lxr L$ with $h=a'p'$. Now consider the pushout
\[
\begin{tikzcd}
E' \arrow[d, "p'"] \arrow[rr, rightarrowtail, "(-q \ \, e')^t"] && E\oplus P \arrow[d] \arrow[rr, twoheadrightarrow, "(e \ \, p)"] && A \arrow[d, equal]  \\
A' \arrow[rr, rightarrowtail] && C \arrow[rr, twoheadrightarrow] && A  
\end{tikzcd} 
\]
in $\smod(\PP)$. By Proposition~\ref{p13}, it follows that $C$ is in $\A$. Since we have the equation $(f\;g')\binom{-q}{e'}=ah=aa'p'$, by the universal property of pushout we get that $(f\;g')$ factors through $E\oplus P\lxr C$. This implies that $f$ factors through $C$.

Next, assume that $M,N\in\varinjlim\A$, and let $f\colon E\lxr L$ be a 
morphism with $E\in\smod(\PP)$. So there are morphisms $e\colon E\lxr A$
and $g\colon A\lxr M$ with $A\in\A$ and $ge=af$. Since $bg$ factors through 
the cokernel of $e$, there is a map $\lambda\colon \Coker{e}\lxr N$ such that $\lambda\pi=bg$. Since $E$ and $A$ belong to $\smod(\PP)$, it follows that $\Coker{e}$ lies also in $\smod(\PP)$. Thus, by Lemma~\ref{l5} the map $\lambda$ factors through an object $A'$. Then we have the following exact commutative diagram$\colon$
\[
\begin{tikzcd}
E \arrow[rr, "e"] \arrow[dd, "f"] && A \arrow[dd, "g"] \arrow[dr, dashed, "a'"]  \arrow[rr, twoheadrightarrow, "\pi"] && \Coker{e} \arrow[dd, "\lambda"]  \arrow[ld, "h'"] \\
  &&  & A' \arrow[dr, "h"] & \\
L \arrow[rr, rightarrowtail, "a"] &&  M \arrow[rr, twoheadrightarrow, "b"] && N   
\end{tikzcd} 
\]
where $\lambda=hh'$, $a'=h'\pi$ and $a'e=0$. Choose a deflation $p\colon P\tra A'$ with $P\in\PP$, and consider the pullback 
\[
\begin{tikzcd}
C \arrow[d, "c"] \arrow[r, twoheadrightarrow, "q"] & A \arrow[d, "a'"] \\
P \arrow[r, twoheadrightarrow, "p"] & A'   
\end{tikzcd} 
\]
in $\A$. So there is a morphism $e'\colon E\lxr C$ with $qe'=e$ and $ce'=0$. Since $P\in\PP$, we find a morphism $r\colon P\lxr M$ with $hp=br$. Thus $b(gq-rc)=0$, which yields a morphism $s\colon C\lxr L$ with $gq-rc=as$. Hence $a(f-se')=ge-(gq-rc)e'=ge-gqe'=0$, and thus $f=se'$. Now  
$\Add\PP\subs\Add\A\subs\varinjlim\A$ implies that 
$\varinjlim\A$ is a resolving subcategory of $\Mod(\PP)$, and consequently, 
$\ke(\A)\subs\varinjlim\A$.
\end{proof}
\end{prop}

\begin{cor}
\label{corkecapsmodP}
Let $\A$ be a skeletally small Ext-category. Then
\[
\ke(\A)\cap \smod(\PP)\simeq {\bf Add}\A.
\]
\begin{proof}
For any object $M$ in $\ke(\A)\cap \smod(\PP)$, we have that $M$ lies in
$\varinjlim\A$. Then, by Lemma~\ref{l5} it follows that the identity map $\iden_M$ factors through an object of $\A$. We infer that $M$ belongs to ${\bf Add}\A$.
\end{proof}
\end{cor}

\begin{defn}
\label{d8}
Let $\A$ be a skeletally small Ext-category with $\PP:=\Proj(\A)$ left coherent. We define 
\begin{enumerate}
\item the category of {\bf accessible big $\A$-objects} to be 
$\ke(\A)$, and 

\item the category of {\bf big $\A$-objects} to be 
$\mt(\ke(\A))$.
\end{enumerate}
\end{defn}  

We clarify below the above two notions of big objects.

\begin{rem}
\label{d8remark}
\begin{enumerate}
\item By Proposition~\ref{p20} and Corollary~\ref{corTTA}, $\mt(\ke(\A))$ is the largest Ext-category containing ${\bf Add\A}:=\Proj(\Mod(\A))$ with ${\bf Add}\Proj\A$ and ${\bf Add}\Inj\A$ as subcategories of projectives and injectives, respectively, while $\ke(\A)$ is the smallest Ext-category with this property. Note that $\mt(\ke(\A))$ contains $\mt(\A)$, which is usually bigger than ${\bf Add}\A$. Thus, in view of Proposition~\ref{proplimresolving} and Corollary~\ref{corkecapsmodP}, $\ke(\A)$ seems to be a better choice than  
$\mt(\ke(\A))$ as a category of big $\A$-objects.

\item Let $R$ be a complete regular local ring, and let $\Lambda$ be a 
Cohen-Macaulay $R$-order. In this context, accessible big Cohen-Macaulay modules according to Definition~\ref{d8} are just what they ought to be (\cite[Section~2]{Lat}), that is, every accessible big $\CM(\Lambda)$-object is $R$-free. This is clear since every object of $\ke(\CM(\Lambda))$ is $R$-free. 

\item In the commutative setting, Griffith \cite[Theorem 1.2, Corollary 2.4]{Gri} proved that if $R$ is a local regular ring and $B$ a module finite extension of $R$ which is Gorenstein, then a $B$-module $C$ is free as an $R$-module if and only if  $C$ is an infinite sygyzy as a $B$-module. The latter means precisely that $C$ is Gorenstein-projective as a $B$-module.
\end{enumerate}
\end{rem}

\begin{rem}
The name ``accessible big'' $\A$-objects is due to its constructive nature. We remark that not all big objects are accessible, for instance big Gorenstein projective modules provide such examples. Let us explain further with an example.

Let $R$ be a self-injective ring. Then all $R$-modules are Gorenstein-projective. Consider the exact category $\A$ of finitely generated projectives/injectives with the trivial exact structure. 
Then, from Proposition~\ref{proplimresolving} the category $\ke(\A)$ is contained in the direct limit category $\varinjlim\A$. Recall that the Govorov-Lazard Theorem says that the closure under direct limit of the class of finitely generated projective modules is equal to the class of flat modules. Hence, in our case we obtain that $\ke(\A)$ is contained in the category of flat $R$-modules. However, we know that not every module may be flat. 
\end{rem}

\begin{exam}
\label{examGproj}
Let $R$ be a ring. By Propositions~\ref{p17} and \ref{p18}, the Gorenstein-projective $R$-modules $\GProj(R):=\mt(\Add{R})=\mt(\ke({\add}R))$ form a resolving subcategory of $\Mod(R)$ (see subsection~\ref{subsectiongorprojectives}). Hence $\GProj(R)$ consists of the big $\add{R}$-objects. If $R$ is left coherent, the category of finitely generated Gorenstein-projective $R$-modules is 
$\Gproj(R):=\mt(\add{R})$. It is a Frobenius category 
with ${\Proj}(\Gproj(R))=\Inj(\Gproj(R))=\add{R}$. Thus 
\[
\GProj(R)=\mt(\ke(\add{R}))=\mt(\ke(\Gproj(R)))
\]
consists of the big $\Gproj(R)$-objects.
\end{exam}

We are now ready to prove the first main result of this section, which gives a general version of Auslander's theorem \cite{Au76} in terms of the category $\ke(\A)$.

\begin{thm}
\label{t3}
Let $\A$ be a left Ext-category. Assume that every object of $\ke(\A)$ is a direct sum of objects in $\A$. Then $\A/[\PP]$ is strongly left noetherian.
\begin{proof}
Consider a family of morphisms $f_{\gamma}\colon A_{\gamma}\lxr A$ in $\A$ with $\gamma\in \Gamma$. Let $p\colon P\lxr A$ be a deflation in $\A$ with $P$ in $\PP$. If $M:=P\oplus \coprod_{\gamma\in \Gamma}A_{\gamma}$, the maps $f_{\gamma}$ together with $p$ give a deflation $q\colon M\lxr A$ in $\ke(\A)$. Thus, we have the following commutative diagram with exact rows in $\ke(\A)\colon$
\begin{equation}
\label{eqthmpushoutsq}
\begin{tikzcd}
A' \arrow[d, rightarrowtail] \arrow[r, rightarrowtail, "j"] & P \arrow[d, rightarrowtail, "i"] \arrow[r, twoheadrightarrow, "p"] & A \arrow[d, equal]  \\
L \arrow[r, rightarrowtail] & M \arrow[r, twoheadrightarrow, "q"] & A  
\end{tikzcd} 
\end{equation}
Note that $i=(1 \ \, 0)^t$, $q=(p, \ \, \coprod_{\gamma\in \Gamma} f_{\gamma})$ and the morphism $A'\lxr L$ is an inflation. From our assumption, the object $L$ is a direct sum of objects in $\A$. Since $A'$ lies also in $\A$, there is a decomposition $L=A_0\oplus L'$ with $A_0$ in $\A$ and $L'$ in $\Add\A$ such that $A'\lxr L$ factors through $A_0\lxr L$. In fact, there is an inflation $a\colon A'\lxr A_0$ such that the following diagram commutes$\colon$
\[
\begin{tikzcd}
A' \arrow[d, rightarrowtail, "a"] \arrow[r, rightarrowtail, "(a \ \, 0)^t" ] & A_0\oplus L'   \\
A_0 \arrow[ru, rightarrowtail, "(1 \ \, 0)^t"'] & 
\end{tikzcd} 
\]
We write $(i' \ \, j')$ for the inflation $A_0\oplus L'\lxr M$. Since the square on the left hand side of diagram $(\ref{eqthmpushoutsq})$ is a pushout, we have the following conflation in $\ke(\A)\colon$
\[
\begin{tikzcd}
A' \arrow[rr, rightarrowtail, "(-j \ \, a \ \, 0)^t"] && P\oplus A_0\oplus L'  \arrow[rr, twoheadrightarrow, "(i \ \, i' \ \, j')"] && M 
\end{tikzcd}
\]
This shows that the map $(0 \ \, 0 \ \, 1)\colon P\oplus A_0\oplus L'\lxr L'$ factors through $(i \ \, i' \ \, j')$, that is, there is a morphism $g\colon M\lxr L'$ with $g\circ (i \ \, i' \ \, j')=(0 \ \, 0 \ \, 1)$. Using this, we get that $(\iden_{M}- j'g)\circ (i \ \, i' \ \, j')=(i \ \, i' \ \, 0)$. Since $A_0$ is in $\A$, there is a finite subset $\Delta\subset \Gamma$ such that the map $i'\colon A_0\lxr M$ factors through the inflation $P\oplus \coprod_{\delta\in \Delta}A_{\delta}\lxr P\oplus \coprod_{\gamma\in \Gamma}A_{\gamma}=M$. Hence, the map $\iden_{M}- j'g$ has its image in $P\oplus \coprod_{\delta\in \Delta}A_{\delta}$. Since $q=q(\iden_{M}- j'g)$, this proves that $\A/[\PP]$ is strongly left noetherian. 
\end{proof}
\end{thm}

The above mentioned hitherto known extensions of Auslander's theorem \cite{Au76} and its converse \cite{RT} are contained in the following corollaries. We start with the following corollary on the artinian case.

\begin{cor}\label{corBelartinianring}
\textnormal{(\cite[Theorem~3.1]{Bel})}
\label{corBel1}
Let $R$ be a be a left artinian ring, and let $\A$
be a resolving subcategory of $\smod{R}$. Then the following are equivalent$\colon$

\begin{enumerate}
\item $\ind(\A)$ is finite.

\item Every flat object in $\Mod(\A)$ is projective.

\item Every object of $\varinjlim\A$ admits a decomposition into objects of 
$\A$.

\item $\ke(\A)={\bf Add}\A$. 
\end{enumerate}        
\begin{proof}
(i)$\Longrightarrow$(ii)$\colon$For $A=\bigoplus \ind(\A)$, consider the ring $S=\End_{\A}(A)^{\op}$. Since finitely generated modules over a left artinian ring have finite length, it follows from \cite[Corollary~29.3]{AF} that $S$ is semiprimary, hence perfect. So Bass' Theorem~P (\cite[Theorem~28.4]{AF}) implies that every flat object in $\Mod(\A)\simeq \Mod(S)$ is projective. 

(ii)$\Longrightarrow$(iii)$\colon$Since $\Add\A$ is dense in $\Mod(R)$ \cite[Chapter~X, Section~6]{MacLane}, there is a full embedding $\Mod(R)\lxr \Mod(\A)$. Then, Lemma~\ref{l5} shows that the $R$-modules in $\varinjlim\A$ are flat in $\Mod(\A)$. Hence, we deduce that $\varinjlim\A=\Add\A$. Since $\A$ is a Krull-Schmidt category, the claim in (iii) follows.

(iii)$\Longrightarrow$(iv)$\colon$This implication follows from Proposition~\ref{proplimresolving}.

(iv)$\Longrightarrow$(i)$\colon$By Corollary~\cite[Corollary~2.6]{KrauseSolberg}, the full subcategory $\A$ is covariantly finite in $\smod(R)$. Hence, $\A$ is an Auslander-Reiten category by \cite[Theorem~2.4]{AuSmsubcat}. By the Crawley-J{\o}nsson-Warfield theorem (see \cite[Theorem~26.5]{AF}), every object of ${\bf Add}\A$ is a direct sum of objects in $\A$. Then Theorem~\ref{t3} implies that $\A/[\add{R}]$ is strongly left noetherian. By Corollary~\ref{CorRumpCor1}, $\mathsf{M}(\A)$ is left L-finite. We infer that $\ind(\A)$ is finite by Theorem~\ref{thmRumpThm2}.
\end{proof}
\end{cor}

As an immediate consequence, Corollary~\ref{corBelartinianring} yields the following restatement.

\begin{cor}\label{corBelartinalg}
\textnormal{(\cite[Corollary~3.5]{Bel})}
\label{corBel2}
Let $\Lambda$ be an artin algebra, and let $\A$
be a resolving subcategory of $\smod(\Lambda)$. Then the following are equivalent$\colon$
\begin{enumerate}
\item $\ind(\A)$ is finite.

\item $\A$ is contravariantly finite in $\smod(\Lambda)$, and any object of $\varinjlim\A$ is a direct sum of objects in $\A$.
\end{enumerate}     
\end{cor}

In the special case where $\A=\add\Lambda$ for an artin algebra $\Lambda$, Example~\ref{examGproj} and Corollary~\ref{corBelartinalg} give the following consequence.

\begin{cor}\textnormal{(\cite[Theorem~4.10]{Bel}} 
\label{corBel3}
Let $\Lambda$ be an artin algebra. The following are equivalent$\colon$ 

\begin{enumerate}
\item Every Gorenstein-projective $\Lambda$-module is a direct sum of finitely generated $\Lambda$-modules.

\item $\ind(\Gproj(\Lambda))$ is finite, and $\GProj(\Lambda)\subset \varinjlim(\Gproj(\Lambda))$.  
\end{enumerate}
\begin{proof}
(i)$\Longrightarrow$(ii)$\colon$By \cite{BelKrause}, the category $\Gproj(\Lambda)$ is contravariantly finite in $\smod(\Lambda)$ and then Corollary~\ref{corBelartinianring} applies.

(ii)$\Longrightarrow$(i)$\colon$It follows by Corollary~\ref{corBelartinalg}.
\end{proof}        
\end{cor}

Recall that a ring $R$ is said to be {\em Gorenstein} if $R$ is noetherian and the injective dimension of $R$ as a left or right $R$-module is finite. The next corollary extends Chen's theorem \cite{Chen} from Gorenstein artin algebras to a wide class of Gorenstein rings. We call a semilocal ring $R$ complete if it is (Hausdorff) complete in its $(\Rad{R})$-adic topology.

\begin{cor}
\label{corChensresult}
Let $R$ be a complete semilocal noetherian Gorenstein ring such that $\add{M}$ is strongly noetherian for all $M$ in $\Gproj(R)$. The following statements are equivalent$\colon$
\begin{enumerate}
\item $\ind(\Gproj(R))$ is finite.

\item Any Gorenstein-projective left or right $R$-module is a direct sum of finitely generated $R$-modules.
\end{enumerate}
\begin{proof}
Recall first that by \cite[Theorem~B]{Rowen}, the category of finitely generated $R$-modules $\smod{R}$ is Krull-Schmidt.

(i)$\Longrightarrow$(ii)$\colon$ We apply Theorem~\ref{thmRumpThm1} to the Frobenius category $\A:=\Gproj(R)$. This means that the stable category $\underline{\A}:=\A/[\add{R}]$ is strongly noetherian and $\Ext(\A)$ is a length category. By \cite[Theorem~6.6]{BelvirtGoralg}, see also \cite[Theorem~4.1]{Chen1}, the triangulated category $\underline{\GProj}(R):=\GProj(R)/[\Add{R}]$ is compactly generated, and its full subcategory of compact elements is equivalent to $\underline{\A}$. By Proposition~\ref{RumpProp2Extab}, there is an equivalence of abelian categories $\smod\underline{\A}\simeq \Ext(\A)$. Thus, for any object $X$ in $\underline{\GProj}(R)$, the $\underline{\A}$-module $X\mapsto \Hom_{\underline{\GProj}(R)}(-,X)|_{\underline{\A}}$ is flat, hence projective. This implies that there are objects $A_i$ in $\A$ with a natural isomorphism$\colon$
\[
\Hom_{\underline{\GProj}(R)}(-, \coprod A_i)|_{\underline{\A}} \cong \Hom_{\underline{\GProj}(R)}(-,X)|_{\underline{\A}}.
\]
Since $A_i\in \underline{\A}$, this isomorphism is induced by a morphism $\coprod A_i\lxr X$ in $\underline{\GProj}(R)$. Consider the triangle $\coprod A_i\lxr X\lxr Y\lxr \coprod A_i[1]$ in $\underline{\GProj}(R)$ and let $A$ an object in $\A$. Then, applying the functor $\Hom_{\underline{\GProj}(R)}(A,-)$ we obtain that $\Hom_{\underline{\GProj}(R)}(A,Y)=0$. Hence $Y=0$, and thus $X\cong \coprod A_i$ in $\underline{\GProj}(R)$. This shows that $\GProj(R)\simeq {\bf Add}\A$. By symmetry, this proves (ii).

(ii)$\Longrightarrow$(i)$\colon$By Theorem~\ref{t3}, the stable category $\underline{\A}$ is strongly left noetherian. Since the functor $\Hom_{\A}(-,R)$ is a duality between $\A$ and $\A^{\op}$, we infer that $\A$ is strongly noetherian. Furthermore, by \cite[Proposition~3.3]{EJX} it follows that $\A$ is contravariantly finite in $\smod(R)$. Thus Theorem~\ref{thmRumpThm1} completes the proof (since $\Gproj(R)$ satisfies the assumption of Theorem~\ref{thmRumpThm1}, see \cite{CMrepresentations, Yoshino}). 
\end{proof}
\end{cor}

As a special case of Corollary~\ref{corChensresult}, we get the following result.

\begin{cor}\textnormal{(\cite[Theorem~4.20]{Bel})}
\label{corBel4}
Let $R$ be a commutative noetherian complete local ring with $\Gproj(R)\neq \add{R}$. The following statements are equivalent$\colon$
\begin{enumerate}
\item $\ind(\Gproj(R))$ is finite.

\item $R$ is Gorenstein, and every Gorenstein-projective $R$-module is a direct sum of finitely generated modules.
\end{enumerate}
\begin{proof}
Any finitely generated $R$-module $M$ admits an epimorphism $R^n\lxr M$, which yields an embedding $\End_R(M)\lxr \Hom_{R}(R^n,M)\cong M^n$. Hence the ring $\End_R(M)$ is noetherian. If (i) holds and since $\Gproj(R)\neq \add{R}$, \cite[Theorem~4.3]{CPST} implies that $R$ is Gorenstein. Thus Corollary~\ref{corChensresult} applies and the result follows.
\end{proof}
\end{cor}

\begin{rem}
Note that the proof of Corollary~\ref{corChensresult} shows that instead of ``$R$ Gorenstein'' in (ii) above, it is enough to assume that the projective and injective object $R\in \Gproj(R)$ is tame (see subsection~\ref{subsectionacyclic}).
\end{rem}

\begin{cor}
Let $\A$ be an Ext-category with right almost split sequences and $\dime\A\leq 2$. Assume that $\add{A}$ is strongly noetherian for each object $A$ in $\A$, and that $\ind(\Proj(\A))$ and $\ind(\Inj(\A))$ are finite. The following are equivalent$\colon$
\begin{enumerate}
\item $\ind(\A)$ is finite.

\item Every big $\A$-object is a direct sum of objects in $\A$.
\end{enumerate}
\begin{proof}
First recall that $\dime\A\leq 2$ implies that every morphism in $\A$ has a kernel, see Remark~\ref{remfindim}.

(i)$\Longrightarrow$(ii)$\colon$ By Theorem~\ref{thmRumpThm1}, the category $\A$ is strongly left noetherian. Thus, by \cite[Theorem~1]{Lat} every morphism in ${\bf Add}\A$ has a kernel. The big $\A$-objects are those of $\ke(\A)$. We infer that $\ke(\A)\simeq {\bf Add}\A$.

(ii)$\Longrightarrow$(i)$\colon$ By Theorem~\ref{t3}, the category $\A/[\PP]$ with $\PP:=\add(\Proj(\A))$ is strongly left noetherian. We show that $\A$ has almost split sequences. Let $A$ be an indecomposable non-projective object of $\A$. Since $\add{A'}$ is strongly noetherian for each object $A'$ of $\A$, there is a morphism $p\colon P\lxr A$ in $\Rad\A$ with $P\in \PP$ such that every morphism $Q\lxr A$ in $\Rad\A$ with $Q\in \PP$ factors through $p$. By Proposition~\ref{p13}, $\A$ is a resolving subcategory of $\smod(\PP)$. So there is a factorization of $p$ in $\smod(\PP)$ as follows$\colon$
\[
\begin{tikzcd}
P \arrow[r, twoheadrightarrow, "q"] \arrow[rr, bend right, "p"']  & J  \arrow[r, rightarrowtail, "j"] & A   
\end{tikzcd} 
\]
Furthermore, the kernel $k\colon K\lxr P$ of $q$ belongs to $\A$. Let $i\colon K\lxr I$ be an inflation in $\A$ with $I$ injective. This gives the following commutative diagram$\colon$
\[
\begin{tikzcd}
K \arrow[d, rightarrowtail, "i"] \arrow[r, "k"]  & P  \arrow[d, rightarrowtail] \arrow[r, "q"] & J \arrow[d, equal] \\
I \arrow[r] & A' \arrow[r, "a"] & J   
\end{tikzcd} 
\]
with short exact rows in $\smod(\PP)$ and $A'$ in $\A$. Now every morphism $f\colon C\lxr A$ in $\Rad\A$ factors through $j$. Since $\Ext^1_{\smod(\PP)}(C,I)=0$, we infer that $f$ factors through $ja$. Hence $ja\colon A'\lxr A$ is right almost split. Thus, $\A$ has almost split sequences. By Corollary~\ref{CorRumpCor1}, $\Ext(\A)$ is a length category. We infer that ${|\ind(\A)|<\infty}$ by Theorem~\ref{thmRumpThm1}.
\end{proof}
\end{cor}

We are now ready to extend the theorems of Ringel and Tachikawa \cite{RT} and Auslander \cite{Au76} to arbitrary dimension. Recall that a Cohen-Macaulay order $\Lambda$ is representation-finite if there are only finitely many isomorphism classes of indecomposable Cohen-Macaulay $\Lambda$-modules (i.e.\ finitely generated $\Lambda$-modules which are free as $R$-modules).

\begin{thm}
\label{t4}
Let $R$ be a complete regular local ring, and let $\Lambda$ be a 
Cohen-Macaulay $R$-order. The following statements are equivalent$\colon$
\begin{enumerate}
\item $\Lambda$ is representation-finite.

\item Every accessible big Cohen-Macaulay $\Lambda$-module is a direct sum of 
finitely generated $\Lambda$-modules.
\end{enumerate}       
\begin{proof}
(i)$\Longrightarrow$(ii)$\colon$For $\dime R=0$, this is Ringel and Tachikawa's theorem \cite{RT}. We proceed by induction on $\dime R$. Thus, assume that $\dime R>0$. Choose any parameter $\pi\in \Rad{R}$. By Theorem~\ref{thmRumpThm1}, $\Ext(\CM(\Lambda))$ is a length category. Hence there is an integer $n\in\mathbb{N}$ with $\pi^n \Ext(\CM(\Lambda))=0$. Let $L$ be a $\Lambda$-module in $\ke_1(\CM(\Lambda))$. So there is a short exact sequence 
$L\rat L_0\tra L_1$ of $\Lambda$-modules with 
$L_0,L_1$ in ${\bf Add}\CM(\Lambda) (:=\Proj(\Mod\CM(\Lambda)))$. This gives a commutative diagram
\[
\begin{tikzcd}
L \arrow[d, "\pi^n"] \arrow[r, rightarrowtail, "i"] & L_0 \arrow[dl, "h"] \arrow[d, "\pi^n"] \arrow[r, twoheadrightarrow, "p"] & L_1 \arrow[d, "\pi^n"]  \\
L \arrow[r, rightarrowtail, "i"] & L_0 \arrow[r, twoheadrightarrow, "p"] & L_1  
\end{tikzcd} 
\]
with a morphism $h\colon L_0\lxr L$ satisfying $hi=\pi^n$. Since $L_0$ lies in
${\bf Add}\CM(\Lambda)$, there is a deflation $d\colon P_0 \tra L_0$ in $\ke(\CM(\Lambda))$ with $P_0$ in ${\bf Add}(\Proj(\CM(\Lambda)))\simeq {\bf Add}{\Lambda}$. By assumption, $\pi^n\colon L_0\lxr L_0$ factors through $d$. Furthermore, $hi=\pi^n$ implies that $\pi^{2n}\colon L\lxr L$ is equal to  
\[
\begin{tikzcd}
L \arrow[r, rightarrowtail, "i"] & L_0 \arrow[r, "\pi^n"] & L_0 \arrow[r, "h"] & L  
\end{tikzcd} 
\]
and therefore $\pi^{2n}\colon L\lxr L$ factors through $P_0$. 

Now let $L'\stackrel{j}{\rat} P\stackrel{q}{\tra} L$ be a short exact sequence of $\Lambda$-modules with $P$ projective. Then $\pi^{2n}\colon L\lxr L$ factors through $q$. So we get a commutative diagram
\[
\begin{tikzcd}
L' \arrow[r, rightarrowtail] \arrow[d, equal] & L'\oplus L \arrow[r, twoheadrightarrow] \arrow[d, rightarrowtail] & L \arrow[d, rightarrowtail,"\pi^{2n}"] \\
L' \arrow[r, rightarrowtail, "j"] & P \arrow[d, twoheadrightarrow] \arrow[r, twoheadrightarrow, "q"] & L \arrow[d, twoheadrightarrow] \\
 & L/\pi^{2n}L \arrow[r, equal] & L/\pi^{2n}L
\end{tikzcd}
\]
of $\Lambda$-modules with short exact rows and columns. By the inductive 
hypothesis, the $R/\pi^{2n}R$-free $\Lambda/\pi^{2n}\Lambda$-module $L/\pi^{2n}L$ is a direct sum of finitely 
generated modules. Hence $L'\oplus L$, as a first syzygy of $L/\pi^{2n}L$,
is stably equivalent to a direct sum of finitely generated $\Lambda$-modules. Thus the $\Lambda$-module $L$ lies in ${\bf Add}\CM(\Lambda)$. This implies that $\ke(\CM(\Lambda))\simeq {\bf Add}\CM(\Lambda)$. 
 As $\CM(\Lambda)$ is a Krull-Schmidt category, the implication (i)$\Longrightarrow$(ii) follows.

(ii)$\Longrightarrow$(i)$\colon$This follows by Theorem~\ref{t3}, Theorem~\ref{thmRumpThm1} (v) and Corollary~\ref{CorRumpCor1} (note that $\CM(\Lambda)$ satisfies the assumptions of the latter two results \cite{Iyama, CMrepresentations, Yoshino}).
\end{proof}
\end{thm}

Let us conclude our paper with the following problem.

\begin{Problem}
Motivated from our Auslander-Ringel-Tachikawa result in higher dimension, it is natural to consider the Auslander-Reiten quiver of a representation-finite Cohen-Macaulay $R$-order where $R$ is a complete $d$-dimensional regular local ring. Igusa-Todorov \cite{IgusaTodorov:1, IgusaTodorov:2, IgusaTodorov:3}, Iyama \cite{Iy0, Iy1, Iy2} and the second author \cite{add} have classified finite Auslander-Reiten quivers in dimension $d\leq 1$. Also, from the work of Reiten-Van den Bergh \cite{ReitenVandenBergh} the Auslander-Reiten quivers are known in dimension 2. For Cohen-Macaulay orders of higher
dimension, only scattered results are known \cite{AR:CMtype, Knorrer, EriksenGustavsen, Solberg}. Enomoto \cite{Enomoto} gives an
Auslander-correspondence for such orders. It seems that with increasing dimension, projectives and injectives play a dominant role in the representation-finite case. We expect that a classification of finite Auslander-Reiten quivers in dimension $3$ will provide a critical knowledge for tackling the problem in all finite dimensions.
\end{Problem}

\end{document}